\documentclass[11pt, psamsfonts]{amsart}
\usepackage{amsmath}
\usepackage{amsthm}
\usepackage{amssymb}
\usepackage{amscd}
\usepackage{amsfonts}
\usepackage{amsbsy}
\usepackage{epsfig}

\newtheorem{maintheorem}{Theorem}
\newtheorem{theorem}{Theorem}[section]
\newtheorem{remark}[theorem]{Remark}
\newtheorem{proposition}[theorem]{Proposition}
\newtheorem{lemma}[theorem]{Lemma}
\newtheorem{corollary}[theorem]{Corollary}
\newtheorem{claim}{Claim}
\newtheorem*{claim*}{Claim}

\def\I{\ensuremath{{\bf 1}}}
\def\eps{\varepsilon}
\def\phi{\varphi}
\def\R{{\mathbb R}}

\def\N{{\mathbb N}}
\def\Z{{\mathbb Z}}

\def\E{{\mathcal E}}

\def\L{{\mathcal L}}
\def\H{{\mathcal H}}

\def\P{{\mathcal P}}
\def\Q{{\mathcal Q}}
\def\F{{\mathcal F}}
\def\D{{\mathcal D}}

\def\S{{\mathcal S}}
\def\T{{\mathcal T}}

\def\es{{\emptyset}}
\def\sm{\setminus}

\def\crit{\mbox{\rm Crit}}

\def\bd{\partial }
\def\le{\leqslant}
\def\ge{\geqslant}
\def\st{such that }

\def\0{\underline 0}

\def\x{\mathrm{X}}
\def\c{\mathrm{C}}
\def\r{\mathrm{R}}

\begin{document}

\title{Statistical stability of equilibrium states for interval
maps}
\author{Jorge Milhazes Freitas}
\author{Mike Todd}

\thanks{JMF is partially supported by POCI/MAT/61237/2004 and MT
is supported by FCT grant SFRH/BPD/26521/2006. Both authors are
supported by FCT through CMUP} 

\begin{abstract}
We consider families of transitive multimodal interval maps with
polynomial growth of the derivative along the critical orbits. For
these maps Bruin and Todd have shown the existence and uniqueness of
equilibrium states for the potential
$\varphi_t:x\mapsto-t\log|Df(x)|$, for $t$ close to $1$. We show that
these equilibrium states vary continuously in the weak$^*$ topology
within such families. Moreover, in the case $t=1$, when the
equilibrium states are absolutely continuous with respect to
Lebesgue, we show that the densities vary continuously within these
families.
\end{abstract}

\maketitle

\section{Introduction}
One of the main goals in the study of Dynamical Systems is to
understand how the behaviour changes when we perturb the underlying
dynamics. In this paper, we examine the persistence of statistical
properties of a multimodal interval map $(I, f)$.   In particular we
are interested in the behaviour of the Ces\`aro means
$\frac1n\sum_{k=0}^{n-1}\phi\circ f^k(x)$ for a potential $\phi:I \to
\R$ for `some' points $x$, as $n \to \infty$. If the system possesses
an invariant \emph{physical measure} $\mu$, then part of this
statistical information is described by $\mu$ since, by definition of
physical measure, there is a positive Lebesgue measure set of points
$x\in I$ such that
\[
\overline\phi(x):=\lim_{n\to\infty}\frac1n\sum_{k=0}^{n-1}\phi\circ
f^k(x)=\int \phi~d\mu.
\]
If for nearby dynamics these measures are proven to be close, then
the Ces\`aro means do not change much under small deterministic
perturbations. This motivated Alves and Viana \cite{AV:2002} to
propose the notion of \emph{statistical stability}, which expresses
the persistence of statistical properties in terms of the continuity,
as a function of the map $f$, of the corresponding physical measures.
A precise definition will be given in
Section~\ref{subsec:statement-results}

However, the study of Ces\`aro means is not confined to the analysis
of these measures.  The study of other classes can be motivated
through the encoding of these statistical properties by `multifractal
decomposition', see \cite{Pesbook} for a general introduction.  Given
$\alpha\in \R$, we define the sets $$B_\phi(\alpha):=\{x\in
I:\overline{\phi}(x) =\alpha\},\ B_\phi':=\{x\in I:\overline{\phi}(x)
\hbox{ does not exist}\}.$$   Then the multifractal decomposition in
this case is
$$I=B_\phi'\cup\left(\bigcup_\alpha B_\phi(\alpha)\right).$$
Understanding the nature of this decomposition gives us information
about the statistical properties of the system.  This can be studied
via `equilibrium states'. See \cite{PesWeiss} for a fuller account of
these ideas, where the theory is applied to subshifts of finite
type.

To define equilibrium states, given a potential $\phi:I \to \R$, we
define the \emph{pressure} of $\phi$ to be
$$P(\phi):=\sup\left\{h_\mu+\int\phi~d\mu\right\},$$
where this supremum is taken over all invariant ergodic probability
measures.  Here $h_\mu$ denotes the metric entropy of the system $(I,
f,\mu)$.  Any such measure $\mu$ which `achieves the pressure', i.e.
$h_\mu+\int\phi~d\mu=P(\phi)$, is called an \emph{equilibrium state}
for $(I, f,\phi)$.

In this paper for a given map $f$, we are interested in the
equilibrium state $\mu_t$ of the `natural' potential $\phi_t:x\mapsto
-t\log|Df(x)|$ for different values of $t$.
For a multimodal map $f$ and an $f$-invariant probability measure $\mu$ we denote the \emph{Lyapunov exponent of $\mu$} by $\lambda(\mu):=\int\log|Df|~d\mu$.
For any $f$ in the class of multimodal  maps $\F$ which we define below, Ledrappier \cite{Ledrap} showed that for $t=1$, there is an equilibrium state $\mu_1$ with
$\lambda(\mu_1)>0$ if and only if $\mu_1$ is absolutely continuous with respect to Lebesgue.  We then refer to $\mu_1$ as an acip.  In this setting any acip is also a physical measure.

Using tools developed by
Keller and Nowicki in \cite{KelNow},  Bruin and Keller \cite{BK} further developed this theory,
showing that for unimodal Collet-Eckmann maps there is an equilibrium
state $\mu_t$ for $\phi_t$ for all $t$ close to 1.  This range of parameters
was extended to all $t$ in a neighbourhood of $[0,1]$ for a special
class of Collet-Eckmann maps by Pesin and Senti \cite{PeSe}. Bruin
and the second author showed similar results for the non-Collet
Eckmann multimodal case in \cite{BTeqnat}.

The \emph{Lyapunov exponent of a point $x\in I$} is defined as $\overline\phi_1(x)$, if this limit exists.  So the set of points with the same Lyapunov exponent is $B_{\phi_1}(\alpha)$.  If $f$ is
transitive and there exists an acip then by the ergodic theorem
$\mu_1(B_{\phi_1}(\lambda(\mu_1))=1$. As shown in \cite{Tmulti},
under certain growth conditions on $f$, for a given value of
$\alpha$, close to $\lambda(\mu_1)$, there is an equilibrium state
$\mu_t$ supported on $B_{\phi_1}(\alpha)$ for some $t$ close to 1.
Therefore, to understand the statistics of the system with potential
$\phi_1$, it is useful to study the properties of the equilibrium states $\mu_t$.

We would like to point out that for continuous potentials $\upsilon:I
\to \R$ the theory of statistical stability has been studied in
\cite{Aru}.  The strength of the current paper is that we deal with
the more natural class of potentials $\phi_t$, which are not
continuous.  By \cite[Theorem 5]{BTeqnat} and \cite[Theorem
6]{BTeqgen} under a growth condition on $f$, there are equilibrium
states for the potential $x\mapsto \phi_t(x)+s\upsilon$ if $\upsilon$
is H\"older continuous and $(t, s)$ is close to $(1, 0)$ or to
$(0,1)$.  For ease of exposition we will not consider this more
general class of potentials here, but we note that the statistical
stability for this more general class equilibrium states can be
proved with only minor changes.

Our method provides a natural framework, via
Gibbs states for inducing schemes, to study convergence of
equilibrium states, broadening and simplifying previous approaches.
Many of our proofs rely strongly on theory developed in \cite{BLS} and \cite{BTeqnat}. Aside from the larger class of measures considered
here, we are also able to study a class of maps with only polynomial
growth along the critical orbit and hence, by \cite{BLS}, whose acips may only have polynomial decay of correlations.  The maps considered here
are multimodal, but we would also like to point out the theory
presented in this paper extends to other cases such as Manneville-Pomeau maps, see Remark~\ref{rmk:M-P}.

\subsection{Statement of results}
\label{subsec:statement-results}

Here we
establish our setting and make our statements more precise.  Let
$\crit=\crit (f)$ denote the set of critical points of $f$. We say
that $c\in \crit$ is a \emph{non-flat} critical point of $f$ if there
exists a diffeomorphism $g_c:\R \to \R$ with $g_c(0)=0$ and
$1<\ell_c<\infty$ \st for $x$ close to $c$,
$f(x)=f(c)\pm|g_c(x-c)|^{\ell_c}$.  The value of $\ell_c$ is known as
the \emph{critical order} of $c$.  We define
$\ell_{max}(f):=\max\{\ell_c:c\in {\rm Crit}(f)\}$. Throughout $\H$
will be the collection of $C^2$ interval maps which have negative
Schwarzian (that is, $1/{\sqrt{|Df|}}$ is convex away from critical
points) and all critical points non-flat.

The \emph{non-wandering set} is the set of points $x\in I$ \st for
arbitrarily small neighbourhoods $U$ of $x$ there exists $n=n(U)\ge
1$ \st $f^n(U)\cap U\neq \es$.  This is the dynamically interesting
set.  As in \cite{HofRai}, for piecewise monotone $C^2$ maps this set
splits into a possibly countable number of sets $\Omega_k$ on which
$f$ is topologically transitive.  As in \cite[Section III.4]{MSbook}
(see also \cite[Section 2.2]{BTeqnat}), these sets $\Omega_k$ can be
Cantor sets or finite sets if the map is renormalisable: that is if
there exists a cycle of intervals $J, f(J), \ldots, f^p(J)$ so that
$f^p(J) \subset J$, $f^p(\bd J) \subset \bd J$.  The other
possibility is that $\Omega_k$ is a cycle of intervals permuted by
$f$.  Our analysis extends to any such cycle of intervals, and indeed
in that case we do not need to make any assumptions on critical
points outside the cycle under consideration.  However, for ease of exposition, we will
assume that  maps in $\H$ are non-renormalisable with only one
transitive component $\Omega$ of the non-wandering set, a cycle of
intervals.  We also assume that for any
$f\in \H$, $f^j(\crit)\cap f^k(\crit)\neq \es$ implies $j=k$.  For
maps failing this assumption, either $f^k(\crit) \cap \crit \neq \es$
for some $k\in \N$, in which case we could consider these relevant
critical points in a block; or some critical point maps onto a
repelling periodic cycle, which we exclude here for ease of
exposition since our method is particularly tailored to case of more
interesting maps where the critical orbits are infinite. It is also
convenient to suppose that there are no points of inflection.

Let $\H_{r,\ell}\subset \H$ denote the set of maps $f\in\H$ with $r$
critical points with $\ell_{max}(f)\le \ell$. We will consider
families of maps in $\H$ which satisfy the following conditions.  The
first one is the Collet-Eckmann condition: For any $r\in \N$,
$\ell\in (1,\infty)$ and $C, \alpha>0$, the class $\F_e(r,\ell,C,
\alpha)$ is the set
\begin{align} f\in \H_{r,\ell} \hbox{ such that } |Df^n(f(c))|
\ge Ce^{\alpha n} \hbox{ for all } c\in \crit \hbox{ and } n\in \N.
\label{eq:CE}
\end{align}
Secondly we consider maps satisfying a polynomial growth condition:
For any $r\in \N$, $\ell\in (1,\infty)$, $C>0$, and any
$\beta>2\ell$, the  class $\F_p(r, \ell, C, \beta)$ is the set
\begin{align}
f\in \H_{r,\ell} \hbox{ such that } |Df^n(f(c))| \ge Cn^\beta \hbox{
for all } c\in \crit \hbox{ and } n\in \N. \label{eq:poly}
\end{align}
We will take a map $f_0\in\F$ where we suppose that either $\F=
\F_e(r,\ell,C, \alpha)$ or $\F=\F_p(r, \ell, C, \beta)$, and consider
the continuity properties of equilibrium states for maps in $\F$ at
$f_0$.

We will consider equilibrium states for maps in these families.
Suppose first that $\F= \F_e(r,\ell, C, \alpha)$.  Then by
\cite[Theorem 2]{BTeqnat}, there exists an open interval
$U_{\F}\subset \R$ containing 1 and depending on $\alpha$ and $r$ so
that for $f\in \F$ and $t\in U_{\F}$ the potential
$\phi_{f,t}:x\mapsto -t\log|Df(x)|$ has a unique equilibrium state
$\mu=\mu_{f}$.  We note that by \cite{MSz}, the fact that there are
$r$ critical points gives a uniform upper bound $\log (r+1)$ on the
topological entropy, which plays an important role in the
computations which determine $U_{\F}$ in \cite{BTeqnat}.  If instead
we assume that $\F=\F_p(r, \ell, C, \beta)$ then by \cite[Theorem
1]{BTeqnat} we have the same result but instead $U_{\F}$ is of the
form $(t_{\F},1]$ where $t_{\F}$ depends on $\beta, \ell$ and $r$.

We choose our family $\F$, fix $t\in U_{\F}$ and denote $\phi_{f,t}$
by $\phi_f$.  For every sequence $(f_n)_n$ of maps in $\F$ we let
$\mu_{n,t}=\mu_{f_n,t}$ denote the corresponding equilibrium state
for each $n$ with respect to the potential $\phi_{f_n}$.   We fix
$f_0\in \F$ and say that $\mu_{0,t}$ is \emph{statistically stable
within the family $\F$}
 if for any sequence $(f_n)_n$ of maps in $\F$ such that
 $\|f_n-f_0\|_{C^2} \to 0$ as $n\to
\infty$, we have that $\mu_{0,t}$ is the weak$^*$ limit of
$(\mu_{n,t})_n$.

\begin{maintheorem}
\label{thm:main} Let $\F\subset\H$ be a family satisfying
\eqref{eq:CE} or \eqref{eq:poly} with potentials $\phi_{f,t}$ as
above.  Then, for every fixed $t\in U_{\F}$ and $f\in\F$, the
equilibrium state $\mu_{f,t}$ as above is statistically stable within
the family $\F$.
\end{maintheorem}

Although the definition of statistical stability involves convergence
of measures in the weak$^*$ topology, when we are dealing with acips,
it makes sense to consider a stronger type of stability due to Alves
and Viana \cite{AV:2002}: for a fixed $f_0\in\F$, we say that the
acip $\mu_{f_0}$ is  \emph{strongly statistically stable} in the
family $\F$ if for any sequence $(f_n)_n$ of maps in $\F$ such that
$\|f_n-f_0\|_{C^2} \to 0$ as $n\to \infty$ we have
\begin{equation}
 \int \left|\frac{d\mu_{f_n}}{dm} -
 \frac{d\mu_{f_0}}{dm}\right|~dm\xrightarrow[n\to\infty]{}0,
 \label{eq:cty}
\end{equation}
where $m$ denotes Lebesgue measure and $\mu_{f_n}$ and $\mu_{f_0}$
denote the acips for $f_n$ and $f_0$ respectively. As a byproduct of
the proof of Theorem~\ref{thm:main} we also obtain:

\begin{maintheorem}
\label{thm:strong-stat-stab} Let $\F\subset\H$ be a family satisfying
\eqref{eq:CE} or \eqref{eq:poly}. Then, for every $f\in\F$, the acip
$\mu_{f}$ is strongly statistically stable.
\end{maintheorem}

For uniformly hyperbolic maps, it is known that the measures do not
merely vary continuously with the map, but actually vary
differentiably in the sense of Whitney.  For example, if $f_0:M\to M$
is a $C^3$ Axiom A diffeomorphism of a manifold $M$ with an unique
physical measure $\mu_0$, and the family $t \mapsto f_t$ is $C^3$,
then the map $t \mapsto \int\psi~d\mu_t$ is differentiable at $t=0$
for any real analytic observable $\psi:M\to \R$, see \cite{Rue}. We
would like to emphasise that the situation for non-uniformly
hyperbolic maps is quite different.  For example if $\F$ is the class
of quadratic maps for which acips exist then it was shown in
\cite{Th:2001} that these measures are not even continuous everywhere in this
family, although as proved in \cite{Ts:1996} they are continuous on a
positive Lebesgue measure set of parameters.  It has been conjectured
in \cite{Bal} that if $\F$ is the set of quadratic maps with some
growth along the critical orbit then the acips should be at most
H\"older continuous in this class, see also \cite{Balopen}. For a
positive result in that direction \cite{RS:1997} proved the H\"older continuity of the densities of the acips as in \eqref{eq:cty} for Misiurewicz parameters. Later, in
\cite{Frei}, strong statistical stability was proved for
Benedicks-Carleson quadratic maps, which are unimodal and satisfy
condition \eqref{eq:CE}. Hence, Theorem~\ref{thm:strong-stat-stab}
provides a generalisation of this last result.

The following proposition shows that the pressure function is
continuous in the family $\F$ with potentials $\phi_{f,t}$.
As with the proof of the main theorems, the proof uses the inducing structure and
also the fact that we have some `uniform decay rate' on these
inducing schemes.  We note that the continuity of the pressure for
continuous potentials was related to statistical stability for
non-uniformly expanding systems in \cite{Aru}.

\begin{proposition}
Let $\F$ satisfy either \eqref{eq:CE} or \eqref{eq:poly} and let
$t\in U_{\F}$.  If $\{f_n\}_{n=0}^\infty\subset \F$ is such that
$\|f_n-f_0\|_{C^2} \to 0$ as $n\to \infty$, then $P(\phi_{f_n,t}) \to
P(\phi_{f_0,t})$ as $n\to \infty$. \label{prop:press cts}
\end{proposition}

\subsection{Structure of the paper}

In Section~\ref{sec:Hofbauer}, we build inducing schemes for each
$f\in\F$ and show that the construction can be shadowed for nearby
dynamics. Although other methods could be used to build the inducing
schemes, we use Hofbauer towers. In
Section~\ref{sec:press}, we introduce some thermodynamic formalism
from  \cite{SaBIP}.  As proved in \cite{BTeqnat}, this gives us the
existence and uniqueness of equilibrium states for our inducing
schemes for the relevant induced potentials.  In particular, these
equilibrium states satisfy the Gibbs property.   In
Section~\ref{sec:gibbs-property} we show that the weak$^*$ limit of
Gibbs measures is also a Gibbs measure.  We also prove
Proposition~\ref{prop:press cts}, the continuity of the pressure over
the family $\F$. In Section~\ref{sec:invariance-weak-limit} this
Gibbs measure is shown to be invariant. Finally, in
Section~\ref{sec:continuity-equilibrium-states} we show that the
continuity of the measures survives the projection of the induced
measures into the original equilibrium states, completing the proof
of Theorem~\ref{thm:main}.  We finish that section by showing that
the choice of inducing schemes and the uniformity properties of the
family $\F$ proved along the way allow us to use the results of
\cite{AV:2002} to obtain Theorem~\ref{thm:strong-stat-stab}.

We emphasise that the main new step in this paper is to use the fact,
proved in \cite{BTeqnat}, that the invariant measures  on our
inducing schemes are Gibbs states.  This allows us to pass
information from the limiting inducing scheme to other nearby
inducing schemes.  In this way we can avoid the techniques of
\cite{AV:2002} which used convergence in the sense of \eqref{eq:cty}.
Those techniques can not be applied directly in this setting since,
unless $\phi=-\log|Df|$, we are not considering acips, and thus
Lebesgue measure has no relevance.

In this paper we write $x=B^\pm y$ to mean $\frac 1B \le \frac xy\le
B$.
For an interval $J$ and a sequence of intervals
$(J_n)_n$, we write $J_n\to J$ as $n\to \infty$ if the convergence
is in the Hausdorff metric.

\emph{Acknowledgements:} We would like to thank Paulo Varandas for
helpful remarks, and Henk Bruin and Neil Dobbs for their comments on
an early version of this paper.  We would also like to thank
anonymous referees for very useful comments.

\section{Choice of inducing schemes}
\label{sec:Hofbauer}

Given $f\in \F$, we say that $(X,F,\tau)$ is an \emph{inducing scheme} for $(I,f)$ if

\begin{list}{$\bullet$}{\itemsep 0.2mm \topsep 0.2mm \itemindent -0mm}
\item $X$ is an interval containing a finite or countable
collection of disjoint intervals $X_i$ \st $F$ maps each $X_i$
diffeomorphically onto $X$, with bounded distortion (i.e. there
exists $K>0$ so that for all $i$ and $x,y\in X_i$, $1/K\le DF(x)/DF(y) \le K$);
\item $\tau|_{X_i} = \tau_i$ for some $\tau_i \in \N$ and $F|_{X_i} = f^{\tau_i}$.
\end{list}
The function $\tau:\cup_i X_i \to \N$ is called the {\em inducing time}. It may
happen that $\tau(x)$ is the first return time of $x$ to $X$, but
that is certainly not the general case.  For ease of notation, we will usually write $(X,F)=(X,F,\tau)$.

Given an inducing scheme $(X,F, \tau)$, we say that a measure $\mu_F$ is a \emph{lift} of $\mu$ if for all $\mu$-measurable subsets $A\subset I$,
\begin{equation} \mu(A) = \frac1{\int_X \tau \ d\mu_F} \sum_i \sum_{k = 0}^{\tau_i-1} \mu_F( X_i \cap f^{-k}(A)). \label{eq:lift}
\end{equation}
Conversely, given a measure $\mu_F$ for $(X,F)$, we say that
$\mu_F$ \emph{projects} to $\mu$ if \eqref{eq:lift} holds.
We call a measure
$\mu$  \emph{compatible to} the inducing scheme $(X,F,\tau)$ if

\begin{list}{$\bullet$}{\itemsep 1.0mm \topsep 0.0mm}\setlength{\itemindent=-7mm}
\item $\mu(X)> 0$ and $\mu\left(X \setminus \left\{x\in X:\tau(F^k(x)) \text{ is defined for all }k\ge 0\right\}\right) = 0$; and
\item there exists a measure $\mu_F$ which projects to $\mu$ by
\eqref{eq:lift}; in particular $\int_X \tau \ d\mu_F <
\infty$.
\end{list}

For the remainder of this paper we will denote the fixed map $f$ in Theorems~\ref{thm:main} and \ref{thm:strong-stat-stab} by $f_0$ and take a sequence $(f_n)_n$ \st
$f_n\in\F$ for all $n$ and $\|f_n-f_0\|_{C^2} \to 0$.

The main purpose this section we use the theory of \emph{Hofbauer towers} developed
by Hofbauer and Keller \cite{Htop,HKerg,Kellift} to produce inducing
schemes as described in \cite{BrCMP}.  We also show how the inducing
schemes move with $n$. Note that we could also have used other
methods to make these inducing schemes, see \cite{BLS} for example.

We let $\Q_{n,k}$ be the natural partition into maximal closed
intervals  on which $f_n^k$ is homeomorphic.  We will denote members
of $\Q_{n,k}$, which we refer to as \emph{$k$-cylinders} by
$\x_{n,k}$.  Note that for $\x_{n,k},\x_{n,k}'\in \Q_{n,k}$ with
$\x_{n,k}\neq\x_{n,k}'$ then $\x_{n,k}\cap\x_{n,k}'$ consists of at
most one point.  For any $x \notin \cup_{j= 0}^kf_n^{-j}(\crit_n)$
there is a unique cylinder $\x_{n,j}$ containing $x$ for $0\le j\le
k$.  We denote this cylinder by $\x_{n,j}[x]$.

We next define the Hofbauer tower.  We let  $$\hat
I_n:=\bigsqcup_{k\ge 0}\hspace{2mm}\bigsqcup_{\x_{n,k}\in \Q_{n,k}}
f_n^k(\x_{n,k})/\sim$$ where $f_n^k(\x_{n,k})\sim
f_n^{k'}(\x_{n,k'})$ if $f_n^k(\x_{n,k})= f_n^{k'}(\x_{n,k'})$.  Let
$\D_n$ be the collection of domains of $\hat I_n$ and $\pi_n:\hat I_n
\to I$ be the inclusion map.  A point $\hat x\in \hat I_n$ can be
represented by $(x,D)$ where $\hat x\in D$ for $D\in \D_n$ and
$x=\pi_n(\hat x)$.  In this case we can also write $D=D_{\hat x}$.

The map $\hat f_n:\hat I \to \hat I$ is defined as
$$\hat f_n(\hat x) = \hat f_n(x,D) = (f_n(x), D')$$
if there are cylinder sets $\x_{n,k} \supset \x_{n,k+1}$ \st $x \in
f_n^k(\x_{n,k+1}) \subset f_n^k(\x_{k,n}) = D$ and $D' = f_n^{k+1}
(\x_{n,k+1})$. In this case, we write $D \to D'$, giving $(\D_n,
\to)$ the structure of a directed graph. We let $D_{n,0}$ denote the
copy of $X_{n,0}=I$ in $\hat I_n$.  For each $R \in \N$, let $\hat
I_n^R$ be the compact part of the Hofbauer tower defined by
$$
\hat I_n^R = \sqcup \{ D \in \D_{n} : \mbox{ there exists a path }
D_{n,0}  \to \dots \to D \mbox{ of length } r \le R \}$$ The map
$\pi_n$ acts as a semiconjugacy between $\hat f_n$ and $f_n$:
$\pi_n\circ \hat f_n=f_n\circ \pi_n$.

A subgraph $(\E, \to)$ of $(\D, \to)$ is called \emph{closed} if $D \in \E$ and $D \to D'$ for some $D'\in \D$ implies that $D' \in \E$. It is \emph{primitive} if for every
pair $D, D' \in \E$, there is a path from $D$ to $D'$ within $\E$.
Clearly any two distinct maximal primitive subgraphs are disjoint. We
define $\D_{n,\T}$ to be the maximal primitive subgraph in $(\D_n, \to)$. We let $\hat I_{n,\T}$ be the union of all of these domains.
This is the \emph{transitive component of $\hat I_n$}.  Since $f\in \F$
is transitive, there is a point $\hat x\in \hat I_{n,\T}$ so that
$\overline{\bigcup_k\hat f_n^k(\hat x)} = \hat I_{n,\T}$.  The
existence and uniqueness of this (maximal) component is implicit in
works of Hofbauer and Raith \cite{HofRai}, see also \cite[Lemma
1]{BTeqnat} for a self contained proof and references for the
existence and uniqueness of this component and the existence of
points with dense orbit.

We next explain how Hofbauer towers can be used to define inducing
schemes.  We use these schemes, rather than, for example, the very
slightly different schemes in \cite{BLS} since these were the schemes
used in \cite{BTeqnat} and so we have good statistical information on
them.  This method of producing inducing schemes was first used in
\cite{BrCMP}.

For an interval $A=(a,a+\gamma)\subset I$ and
$\delta>0$, we let $(1+\delta)A$ denote the interval
$(a-\delta\gamma, a+\gamma+\delta\gamma)\cap I$. Fixing $\delta>0$
once and for all, for $A\subset I$ we let $A'=(1+\delta)A$ and
define
\begin{equation}\check A= \check
A_n(\delta):=\sqcup\{D\cap\pi_n^{-1}(A):D\in \D_n,
\pi_n(D) \supset A'\}.\label{eq:hats}\end{equation}

Following the method of \cite[Section 3]{BrCMP}, we pick some
cylinder $\x_{n,k}\in \Q_{n,k}$ and consider the first return map
$F_{\check \x_{n,k}}:\bigcup_j\hat\r^j \to \check \x_{n,k}$  where
$\check \x_{n,k}$ is derived as in \eqref{eq:hats} and $F_{\check
\x_{n,k}}=\hat f^{r_{\check\x_{n,k}}}$ for the return time $r_{\check
\x_{n,k}}$ which is constant on each of the first return domains $\hat\r^j$. The fact, explained above, that each $\hat f_n$ is topologically transitive on $\hat I_{n,\T}$ implies that $\overline{\bigcup_j\hat\r^j}=\check \x_{n,k}$.
We next can define an inducing scheme $F_{\x_{n,k}}:\bigcup \r^j \to
\x_{n,k}$ with inducing time  $\tau_{\x_{n,k}}(x)= r_{\check
\x_{n,k}}(\hat x)$ for some $\hat x\in  \check \x_{n,k}$ such that
$\pi_n(\hat x)=x$. As shown in \cite[Section 3]{BrCMP}, this number
is the same for any such choice of $\hat x$.  Here, after possibly
relabelling, each $\hat \r^{j}$ is $\pi^{-1}(\r^j)\cap
\check\x_{n,k}$.  Moreover,  $\tau_{\x_{n,k}}$ is constant
$\tau_{\x_{n,k}}^j$ on each $\r^j$.   Let $(\x_{n,k})^\infty$ denote
the set of points for which
$\tau_{\x_{n,k}}(F_{\x_{n,k}^i}^j(x))<\infty$ for all
$j=0,1,\ldots$.

The main result of this section is the following proposition.

\begin{proposition}
Let $\F$ be a fixed family $\F=\F_e(r,\ell,C, \alpha)$ or $\F=\F_p(r,
\ell, C, \beta)$ satisfying  \eqref{eq:CE} or \eqref{eq:poly}
respectively.  We let $(f_n)_n$ be any sequence \st $f_n\in\F$ for all $n$, and $x$ be any point in $\Omega\sm\left(\cup_{j\in \Z}f_0^j(\crit_0)\right)$.  For any $k$, if $n$ is sufficiently large then there exists a
sequence of inducing schemes $(\x_{n,k}[x], F_{\x_{n,k}[x]})$ as
defined above such that $\x_{n,k}[x]\to\x_{0,k}[x]$ in the Hausdorff
metric, and for any $y$ in the interior of any $\r^j$,
$F_{\x_{n,k}[x]}(y) \to F_{\x_{0,k}[x]}(y)$ as $n\to \infty$.
\label{prop:convergent ind}
\end{proposition}

Note that the inducing schemes depend on $\delta>0$, but any choice
will work for all $n$.

Our inducing schemes are created as first return maps to sets $\check
\x_{n,k}$ as defined above.  Since the structure of these sets is
determined by the structure of the Hofbauer tower, in order to prove
the proposition we first have to show that the respective Hofbauer
towers converge. Moreover, since the domains $\check X_{n,k}$ must be
chosen inside $\hat I_{n,\T}$, we need to show that the sets $\hat I_{n,\T}$ converge.  Without our assumptions on the uniform growth of all $f_n$, this may not be the case.

\begin{lemma}
Let $\F$ be a fixed family $\F=\F_e(r,\ell,C, \alpha)$ or $\F=\F_p(r,
\ell, C, \beta)$ satisfying  \eqref{eq:CE} or \eqref{eq:poly}
respectively. Then $\hat I_{n,\T} \to \hat I_{0,\T}$ in the sense
that for any $R\in \N$, $\hat I_{n,\T}^R \to \hat I_{0,\T}^R$ in the
Hausdorff metric. \label{lem:trans conv}
\end{lemma}

The proof of this lemma relies on the properties of measures on the
Hofbauer tower, so before proving it, we show how to find representatives of these measures
on the towers.  Given $f\in \F$, we define $\iota:=\pi|_{D_0}^{-1}$ where $D_0$ is the lowest level in $\hat I$, so $\iota:I \to D_0$ is an inclusion map. Given a probability measure $m$,
let $\hat m^0 = m\circ\iota^{-1}$ be a probability measure on  $D_0$. Let
$$\hat m^k:=\frac1k\sum_{j=0}^{k-1}\hat m^0\circ \hat f^{-j}.$$
We let $\hat m$ be a vague limit of this sequence.  This is a
generalisation of weak$^*$ limit for non-compact sets: for details
see \cite{Kellift}.  In general it is important to ensure that $\hat
m \not\equiv 0$.  It was proved in \cite[Theorem 3.6]{BK} that if $m$
is an ergodic invariant measure with positive Lyapunov exponent then
$\hat m\circ\pi^{-1} = m$.  Note that for $f\in \H$, any $\hat
f$-invariant probability measure $\hat m$, we must have $\hat m(\hat
I_{\T})=1$.

\begin{proof}[Proof of Lemma~\ref{lem:trans conv}]
We first prove the following claim.

\begin{claim}
$\hat I_n \to \hat I_0$ in the sense that for $R\in \N$, $\hat I_n^R
\to \hat I_0^R$ in the Hausdorff metric. \label{claim:Hof conv}
\end{claim}

\begin{proof} The idea of this proof is that since
$\|f_n-f_0\|_{C^2}\to 0$, for any $k\in \N$ the set $\cup_{j=1}^kf_n^{-j}(\crit_n)$ is topologically the same as the set $\cup_{j=1}^kf_0^{-j}(\crit_0)$ for all $n$
large.  Observe that this is not necessarily true if we did not
assume that critical orbits do not intersect, since otherwise there
could be a point $x$ such that $f_0^j(x)\in \crit_0$ and
$f_0^{j'}(x)\in \crit_0$ for $j\neq j'$ and for each $n$ two points
$x_{n,1}\neq x_{n,2}$ such that $f_n^j(x_{n,1})\in \crit_n$ and
$f_n^{j'}(x_{n,2})\in \crit_n$ and $x_{n,1}, x_{n,2} \to x$ as $n\to
\infty$.

This means that for fixed $k$ and for all large $n$, we can define an
order preserving bijection on this set
$h_{n,k}:\cup_{j=0}^kf_n^{-j}(\crit_n) \to
\cup_{j=0}^kf_0^{-j}(\crit_0)$ so that for $x\in I$ such that
$f_n^k(x)\in \crit_n$, $f_0^k\circ h_{n,k}(x)\in \crit_0$ and for all
$0\le j\le k$ we have $h_{n,k}\circ f_n^j(x)= f_0^j\circ h_{n,k}(x)$.
Therefore, for any cylinder $\x_{0,k}^i\in \Q_{0,k}$ for large enough
$n$, there is a corresponding cylinder $\x_{n,k}^i\in \Q_{n,k}$ so
that $\x_{n,k}^i \to \x_{0,k}^i$.  The existence of $h_{n,k}$ also
implies that if, given $R\in \N$, for some $k, k'\le R$, we have
$f_0^k(\x_{0,k}^i)=f_0^{k'}(\x_{0,k'}^{i'})$ then for all large $n$,
for the corresponding cylinders $\x_{n,k}^i$ and $\x_{n,k'}^{i'}$ we
have $f_n^k(\x_{n,k}^i)=f_n^{k'}(\x_{n,k'}^{i'})$.  Hence for fixed
$R\in \N$, $\hat I_n^R$ and $\hat I_0^R$ are topologically the same
for all large $n$.  Since $\|f_n-f_0\|_{C^2}\to 0$ they also
converge in the Hausdorff metric, completing the claim.
\end{proof}

Recall that our assumptions on the transitivity of maps in $ \H$ and
\cite[Lemma 1]{BTeqnat} imply that there is a unique transitive
component in the Hofbauer tower. The following claim will allow us to
compare the transitive components of our Hofbauer towers.

\begin{claim}
There is a domain $D_0^*\in D_{0,\T}$ so that the  corresponding
domains $D_n^*$ are in $D_{n,\T}$. \label{claim:Dstar}\end{claim}

Assuming this claim, we use the fact that $(\D_{n,\T}, \to)$ is a
closed subgraph, i.e. if $D\in \D_{n,\T}$ and there exists a path
$D\to \cdots \to D'$ then $D'\in \D_{n,\T}$.  Let $D_0$ be an
arbitrary domain in $\D_{0,\T}$.  By Claim~\ref{claim:Hof conv}, for
large enough $n$, there exists a corresponding domain $D_n\in \D_n$.
Since there must exist $\hat x_0\in D_0^*$ with dense orbit in $\hat
I_{0,\T}$, this point iterates into $D_0$ and back into $D_0^*$.
Therefore,  for large enough $n$ there is a point $\hat x_n\in D_n^*$
which iterates into the corresponding domain $D_n\in \D_n$ and back
out to $D_n^*$.  Since $D_n^*\in \D_{n,\T}$, we must also have
$D_n\in \D_{n,\T}$.  Therefore, once the claim is proved, so is the
lemma.

\begin{proof}[Proof of Claim~\ref{claim:Dstar}]
For $f\in \F$, and an open interval $U\subset I$, we say that $x$
makes a \emph{good entry to $U$ at time $k$} if there exists an
interval $U'\ni x$ so that $f^k:U'\to U$ is a homeomorphism.
We first show that for the inducing domains $\Delta_n$ constructed as
in \cite{BLS}, there exists $\theta>0$ such that $m$-a.e. $x$ makes a
good entry to $\Delta_n$ under iteration by $f_n$ with frequency at
least $\theta$.  Here $\theta>0$ is uniform in $n$ and
the measure $m$ is Lebesgue.  In fact we are really interested in good entries to a subset of $\Delta_n$, but we do this first since it is simpler and introduces the ideas.

We let $\nu_n$ denote the acip for $(I, f_n)$.  In \cite{BLS} inducing schemes
$G_n:\bigcup \Delta_n^i\to \Delta_n$ are constructed for some
$\Delta_n$.  Here $G_n=f^{r_n}$ for an inducing time $r_n$.  We can
take $\eta_n=\frac{|\Delta_n|}2$.  This is uniformly bounded below,
by some $\eta>0$. To check this fact we refer to \cite[Lemma
4.2]{BLS} where the sets $\Delta_n$ are constructed. Then observe
that once a map $f_0$ is fixed, the construction of the corresponding
$\Delta_0$ involves a finite number of iterations and constants that
can be taken uniformly within a neighbourhood of $f_0$. This means
that one can mimic the construction for a neighbouring map $f_n$ and
hence obtain an interval uniformly close to the original $\Delta_0$.

By the Ergodic Theorem, the
 frequency
$$\lim_{k\to \infty}\frac1k\#\left\{0\le j<k:\exists U\ni x \hbox{
s.t. }
 f^j:U\to \Delta_n \hbox{ is a diffeomorphism}\right\}$$
for a $\nu_n$-typical point $x$ is bounded below by 
$1/\int r_n~d\nu_{G_n}$ where $\nu_{G_n}$ is the measure for the inducing
scheme as in \eqref{eq:lift}.  Since $\nu_n$ is equivalent to Lebesgue, we need only to show
that $\int r_n~d\nu_{G_n}$ is uniformly bounded above for all $f_n\in
\F$, with $n$ sufficiently large.  This fact follows from
Lemma~\ref{lem:uniform decay} later in the paper.

As above, $\Delta_n\to \Delta_0$ in the Hausdorff metric.  This means
that there exists $k$ and cylinders $\x_{n,k}\in \Q_{n,k}$ so that
$\x_{n,k} \to \x_{0,k}$ and for all large $n$, $\x_{n,k}\Subset
\Delta_n$.  We set $D_n^*= f_n^k(\x_{n,k})$.
A similar argument to the one above shows that $m$-a.e. $x$ makes a good entry to $\x_{n,k}$ under iteration by $f_n$ with frequency bounded below by $1/\int_{\x_{n,k}} r_n~d\nu_{G_n}$.  Since $\x_{n,k}$ converge to some $\x_{0,k}$ as in Claim 1, again Lemma~\ref{lem:uniform decay} implies that for all large $n$, $m$-a.e. $x$ makes a good entry to $\x_{n,k}$ under iteration by $f_n$ with frequency bounded below by $\theta'$ where  $\theta'$ is any value in $(0,1/\int_{\x_{0,k}} r_0~d\nu_{G_0})$.

Given $f\in \H$,  the way the Hofbauer tower is constructed using the
cylinder structure means that if $U\subset I$ is an interval such
that $f^k:U \to f^k(U)$ is a homeomorphism, then for $\hat
U:=\iota(U)\subset D_0$, every iterate $\hat f^j(\hat U)$ is
contained in a unique element of $\D$ for $0\le j\le k$.  Moreover $\pi(\hat f^k(\hat U))=f^k(U)$. Therefore,  if $x$ makes a good entry to $\x_{n,k}$ at
time $j$ under iteration by $f_n$, then there exists an interval
$B$ in a unique domain of $\D_n$ so that $\pi_n(B)=\x_{n,k}$ and for $\hat
x:=\iota_n(x)$, $\hat f_n^j(\hat x)\in B$.   For any such a domain $B$, we must have $\hat f_n^k(B) = D_n^*$, and hence $\hat f_n^{j+k}(\hat x)\in D_n^*$.

Fix $\eps>0$.  The above argument means that the frequency of entries
of a point $\iota_n(x)$ to $D_n^*$ under iteration by $\hat f$ can be
estimated in terms of good entries of $x$ to $\x_{n,k}$ under iteration by $f_n$.  Hence there
exists $k_0=k_0(n,x,\eps)\in \N$ so that
 $k\ge k_0$ implies
$$\frac1k\#\left\{0\le j<k:\hat f^j(\hat x)\in \hat D_n^*\right\}>
\frac{\theta'}{1+\eps}.$$ Let $N$ be so large that $m\{x\in
I:k_0(n,x,\eps)\le N\}\ge 1-\eps$.
 Then
$$\hat m^k(D_n^*)= \frac1k\sum_{j=0}^{k-1}\hat m^0\circ
\hat f^{-j}(D_n^*) \ge \theta' \left(\frac{1-\eps}{1+\eps}\right)$$
for all $k\ge N$.  Since $\eps>0$ was arbitrary, we have $\hat
m(D_n^*)\ge \theta'$ for all large $n$.  Since $\hat m$ can only give
mass to domains in transitive components, this implies that
$D_n^*\subset \hat I_{n,\T}$ for all large $n$.
\end{proof}
\end{proof}

\begin{proof}[Proof of Proposition~\ref{prop:convergent ind}.]
Lemma 1 of \cite{BTeqnat} implies that for any point $x\in \Omega\sm
\left(\cup_{k\in \Z}f_0^k(\crit_0)\right)$ there exists $\hat x\in
\hat I_{0,\T}$ so that $\pi_0(\hat x)=x$.  Hence for large enough
$k$, the cylinder $\x_{0,k}[x]$ will give rise to a set $\check
\x_{0,k}[x]\subset \hat I_{0,\T}$ as in our construction which is
non-empty.  By Lemma~\ref{lem:trans conv}, this will also be true of
the corresponding cylinder for $f_n$ for all $n$ large enough.
Moreover, that lemma implies that for any $R\in \N$,  $\check
\x_{n,k}[x]\cap \hat I_n^R \to \check \x_{0,k}[x]\cap \hat I_0^R$ as
$n\to \infty$.  Hence, the first return map by $\hat f_n$ to $\check
\x_{n,k}[x]$ converges pointwise to that of $\hat f_0$ to $\check
\x_{0,k}[x]$. Therefore for any $y\in \cup_j \stackrel{\circ}{R^j}$,
$F_{\x_{n,k}[x]}(y) \to F_{\x_{0,k}[x]}(y)$ as $n\to \infty$, as
required.
\end{proof}

\section{Equilibrium states for the induced maps}
\label{sec:press}

For a dynamical system $T:X\to X$ on a topological space and $\Phi:X
\to \R$, we say that a measure $m$ is {\em $\Phi$-conformal}
(and call $\Phi$ a potential) if $m(X)=1$ and
\[
m(T(A)) = \int_A e^{-\Phi(x)}~dm(x)
\]
whenever $T:A \to T(A)$ is one-to-one. In other words, $dm\circ T(x)
= e^{-\Phi(x)}dm(x)$.

Assume that $\S_1 = \{\c_1^i \}_i$ is a countable Markov partition of
$X$ \st $T:\c_1^i \to X$ is injective for each $\c_1^i \in \S_1$. We
denote $\S_k := \bigvee_{j=0}^{k-1} T^{-j}(\S_1)$, the $k$-join of
the Markov partition $\S_1$, and suppose that the topology on $X$ is
generated by these sets.  We define\begin{equation}\label{eq:var}
V_k(\Phi) := \sup_{\c_k \in \S_k} \sup_{x,y \in \S_k} |\Phi(x) -
\Phi(y)|,
\end{equation}
We say that $\Phi$ has \emph{summable variations}  if $\sum_{k\ge
1}V_k(\Phi)<\infty$.

We define the transfer operator for a potential $\Phi$ with summable
variations as
\[
(\L_{\Phi}g)(x) := \sum_{T(y) = x} e^{\Phi(y)} g(y),
\]
where $g$ is in the Banach space of bounded continuous complex valued
functions.

Suppose that $(X,T)$ is topologically mixing and $\Phi$ is a
potential with summable variations.  For every $\c_1^i \in \S_1$ and
$k\ge 1$ let
\[
Z_k(\Phi, \c_1^i):=\sum_{T^k x = x}e^{\Phi_k(x)}1_{\c_1^i}(x),
\]   where $\Phi_k(x)= \sum_{j=0}^{k-1} \Phi \circ T^j(x)$.    As in
\cite{Satherm}, we define the \emph{Gurevich pressure} of $\Phi$ as
\begin{equation*}\label{eq:Gur}
P_G(\Phi) := \lim_{k \to \infty}\frac 1k \log Z_k(\Phi, \c_1^i).
\end{equation*}
This limit exists since $\log Z_k(\Phi, \c_1^i)$ is almost
superadditive: \begin{equation*} \log Z_k(\Phi, \c_1^i)+\log
Z_j(\Phi, \c_1^i)\le \log Z_{k+j}(\Phi, \c_1^i) + \sum\nolimits_{n\ge
1}V_n(\Phi).
\label{eq:superadd} \end{equation*}  Therefore,
$P_G(\Phi)=\sup_{k}\frac1k \log Z_k(\Phi, \c_1^i)> -\infty$.  By the
mixing condition, in \cite[Lemma 3]{Satherm}, $P_G(\Phi)$ is
independent of the choice of $\c_1^i$. To simplify the notation, we
will often suppress the dependence of $Z_k(\Phi,\c_1^i)$ on $\c_1^i$.
Furthermore, if $\|\L_\Phi 1\|_\infty<\infty$ then
$P_G(\Phi)<\infty$, see \cite[Lemma 2]{Satherm}.

Assume now that $T:X \to X$ is the full shift.  That is $T:\c_1^i\to
X$ is bijective for all $i$.

We say that $\mu$ is a \emph{Gibbs measure} if there exists
$K<\infty$  \st for all $\c_k\in \S_k$,
$$\frac1{K} \le \frac{\mu(\c_k)}{e^{\Phi_{k}(x)-kP_G(\Phi)}}\le K$$
for any $x\in \c_k$.  Here $\Phi_{k}(x):= \Phi(T^{k-1}(x))+\cdots +
\Phi(x)$.
\begin{theorem}[\cite{SaBIP}] If $(X,T)$ is the full shift, $\Phi:X
\to \R$ is a potential with $\sum_{k\ge 1}V_k(\Phi)<\infty$ and
$P_G(\Phi)<\infty$ then
 \begin{itemize}
\item[(a)] There exists a unique Gibbs measure $m_\Phi$, which
    is
    moreover $(\Phi - P_G(\Phi))$-conformal;
    \item[(b)] There exists an invariant probability measure, which is also Gibbs,
    $\mu_\Phi=\rho_\Phi m_\Phi$ where $\rho_\Phi$ is unique so
    that $\L_\Phi \rho_\Phi=e^{P_G(\Phi)} \rho_\Phi$.  Moreover,
    $\rho_\Phi$ is positive, continuous and has summable
    variations;
    \item[(c)]
 If $h_{\mu_\Phi}(T) < \infty$ or $-\int \Phi d\mu_\Phi
< \infty$ then $\mu_\Phi$ is the unique equilibrium state (in
particular, $P(\Phi) = h_{\mu_\Phi}(T) + \int_X \Phi~d\mu_\Phi$);
\item[(d)] The Variational Principle holds, i.e.,
$P_G(\Phi)=P(\Phi)$ ($=h_{\mu_\Phi}(T) + \int_X
\Phi~d\mu_\Phi$).
\end{itemize}
\label{thm:BIP}
\end{theorem}
Note that because $\mu_\Phi$ is a Gibbs measure, $\mu_\Phi(\c_k^i) >
0$ for every cylinder set $\c_k^i \in \S_k$, $k \in \N$.

From Proposition~\ref{prop:convergent ind}, we have inducing schemes
$(\c_{n,0},F_n)$ for $\c_{n,0}=\x_{n,k}^i$ and $F_n=F_{\x_{n,k}^i}$.
As in \cite{BTeqnat} we fix $t\in U_{\F}$ and let
$\psi_n=\psi_{n,t}:=\phi_{f_n,t}-P(\phi_{f_n,t})$.  The corresponding
induced potential is defined as $\Psi_n(x)=\Psi_{F_n}(x):=\psi_n\circ
f^{\tau_n(x)}(x)+\cdots +\psi_n(x)$.  These schemes can be coded
symbolically by the full shift on countably many symbols.  We denote
a $k$-cylinder of $F_n$ by $\c_{n,k}$, and the collection of these
cylinders by $\P_{n,k}$.  This plays the role of $\S_k$ in the
discussion of the full shift above.  We denote $\Psi_{n,k}(x):=
\Psi_n(F_n^{k-1}(x))+\cdots +\Psi_n(x)$.
The variation $V_{k}(\Psi_n)$ is defined as in \eqref{eq:var}.

As shown in \cite{PeSe, BTeqnat}, Theorem~\ref{thm:BIP} can then  be used to
produce equilibrium states for the systems $(\c_{n,0},F_n,\Psi_{n})$.
Firstly, it can be shown, for example in \cite[Lemma 7]{BTeqnat}, that
$\Psi_n$ have summable variations.  Next, in the proofs of Theorem 1
and 2 of \cite{BTeqnat} it was proved that $P_G(\Psi_n)=0$ when $f_n$
satisfies \eqref{eq:poly} and \eqref{eq:CE} respectively. Then using
Theorem~\ref{thm:BIP} we get conformal measures $m_{F_n}=m_{\Psi_n}$,
densities $\rho_{F_n}=\rho_{\Psi_n}$, and
equilibrium states $\mu_{F_n}=\rho_{F_n}m_{F_n}$ for $(\c_{n,0},
F_n,\Psi_n)$.  These project to equilibrium states
$\mu_n=\mu_{\psi_n}$ for $(I, f_n,\psi_n)$.  Note that an equilibrium
state for $\psi_n$ is also an equilibrium state for $\phi_{f_n,t}$.  Also note that these arguments imply that $\mu_{f_n}$ is compatible to each of the inducing schemes in Proposition~\ref{prop:convergent ind}.

We finish this section by proving a uniform bound on the variation of
$\rho_{F_n}$ which will be useful later.

\begin{remark}
We define the distortion constants $$B_{n,k}:= \exp\left(\sum\nolimits_{j\ge
k+1}V_j(\Psi_{n})\right).$$  By \cite[Lemma 7]{BTeqnat} the Koebe space given by the fact that our inducing schemes have diffeomorphic
extensions to $(1+\delta)C_{n,0}$ implies that there exist
$0<\lambda(\delta, t)<1$ and $C(\delta)>0$ so that $V_k(\Psi_n) \le
C(\delta)\lambda(\delta,t)^k$.  Then there exist $C'(\delta)>0$ and
$\lambda'(\delta,t)$ so that $B_{n,k}=B_{n,k}(\delta,t) \le
C'(\delta)\exp\left(\sum_{j\ge k+1}\lambda'(\delta,t)^k\right)$.
Therefore we can choose $B_{n,k}$ to be independent of $n$.  We
denote this bound by $B_k$. \label{rmk:distn}
\end{remark}

Following the proof of \cite[Theorem 1]{SaBIP}, any constant $H_n$
with  $H_n \ge (\sup \rho_{F_n})^2$ where $\rho_{F_n}$ is as in
Theorem~\ref{thm:BIP}(b) has the following property.  For any
$\c_{n,k}\in \P_{n,k}$,
$$\frac1{H_nB_0} \le
\frac{\mu_{F_n}(\c_{n,k})}{e^{\Psi_{n,k}(x)-kP_G(\Psi_n)}}\le
H_nB_0$$ for any $x\in \c_{n,k}$.  As noted above, by \cite{BTeqnat},
$P_G(\Psi_n)=0$. We are allowed to take a uniform distortion constant
$B_0$ for all of our maps $F_n$ by our choice of $\c_{n,0}$. It is
important here to replace $H_n$ with a uniform constant $H$.  We
consider how $H_n$ was obtained.  For the following lemma and its
proof we fix $f=f_n$, so dropping any extra notation. Note that the
bound used in \cite[p1754]{SaBIP} is not sufficient for us since it
depends on the measure of a cylinder, which can be different for
different $n$.

\begin{lemma}
$V_0(\log \rho_\Psi) \le 2\log B_0$ and the Gibbs constant can be
chosen to be $H_f=(B_0)^4$. \label{lem:H}
\end{lemma}

\begin{proof}
According to \cite[(3.12)]{Sathesis}, see also the argument of
\cite[Lemma 6]{Satherm}, \linebreak $V_1(\log\rho_\Psi)<\log B_1$. We
use this to show $V_0(\log\rho_\Psi)$ is uniformly bounded above.
Take $x_1,x_2\in \c_0$.  Let $y_{1,i}, y_{2,i}$ be the unique points
in $ \c_1^i$ \st $F(y_{1,i})=x_1$ and $F(y_{2,i})=x_2$.  Then since
$\L_\Psi \rho_\Psi=\rho_\Psi$,
\begin{align*}
\left|\frac{\rho_\Psi(x_1)}{\rho_\Psi(x_2)}\right| & = \left|
\frac{\sum_{Fy_1=x_1} e^{\Psi(y_1)}\rho_\Psi(y_1)}{\sum_{Fy_2=x_2}
e^{\Psi(y_2)}\rho_\Psi(y_2)}\right| \le \left|\frac{\sum_i \sup_{x\in
\c_1^i} e^{\Psi(x)}\rho_\Psi(x) }{\sum_i \inf_{y\in \c_1^i}
e^{\Psi(y)}\rho_\Psi(y)}\right|\le B_0B_1.
\end{align*}
 Therefore the first part of the lemma is finished.
There must exist $x_1, x_2\in \c_0$ with $\rho_\Psi(x_1) \le 1$ and
$\rho_\Psi(x_2)\ge 1$: otherwise in the first case $\mu_F(\c_0)>1$,
and in the second case $\mu_F(\c_0)<1$.  So setting $H_f:=(B_0)^4$ we
have $H_f \ge (\sup \rho_\Psi)^2$, so we are finished.
\end{proof}

For use later, we let $H_{\F}:=(B_0)^4$.

\section{Gibbs property for the weak$^*$ limit of Gibbs measures}
\label{sec:gibbs-property}

Later in this section we will fix some inducing schemes $(X_n,F_n)$
as in Proposition~\ref{prop:convergent ind} with induced measures
$\mu_{F_n}$. By passing to a subsequence if necessary, we can define
$\mu_{F_\infty}$, a weak$^*$ limit of $\left(\mu_{F_n}\right)_{n}$.
From the previous section and a uniqueness argument from
\cite{MauUrb}, we know that if we prove that $\mu_{F_\infty}$
satisfies the Gibbs property and is invariant, then $\mu_{F_\infty} =
\mu_{F_0}$. This section is devoted to proving that $m_{F_\infty}$
has the Gibbs property which will allow us to conclude that
$\mu_{F_\infty}$ has the Gibbs property also.  The proof of the
following lemma relies heavily on \cite{BLS} and \cite{BTeqnat}.  In
the proof we outline the main ideas used from those papers, in
particular the origin of the important constants used.

\begin{lemma}
\label{lem:uniform decay} For a given family $\F=\F_e(r,\ell,C,
\alpha)$ or $\F=\F_p(r, \ell, C, \beta)$ satisfying  \eqref{eq:CE} or
\eqref{eq:poly} respectively, and every $f$ in a neighbourhood of any
$f_0\in\F$ fixed, there exist $C'>0$, $\alpha'>0$ or $\beta'>0$ and
an inducing scheme $(X,F)$ as in Proposition~\ref{prop:convergent
ind}, with inducing time $\tau$, so that for all $N\ge 1$,
$$\mu_{F}\{\tau>N\} \le C'e^{-\alpha'N} \hbox{ or }
\mu_{F}\{\tau>N\} \le C'N^{-\beta'} \hbox{ respectively }.$$
\end{lemma}
\begin{proof}
In the proof of Theorem 1 of \cite{BTeqnat}, the correspondence
between our inducing scheme and the one considered in \cite{BLS} is
given, which allows to conclude that the estimates for the tail of
our inducing scheme in the case the potential is $\psi_{n,1}$, i.e. for the acip, are given by the ones in \cite{BLS}.  In the proofs of Theorems 1 and 2 of \cite{BTeqnat} it is then shown how estimates for the tails for the potentials $\psi_{n,t}$ can be easily related.  The main
result in \cite{BLS} is proven by establishing that the growth rate
of the derivative along the critical orbits determines the rate at
which the Lebesgue measure of the tail of the inducing scheme falls
off with time. Hence, roughly speaking, the estimates for the tail
obtained in \cite{BLS} depend essentially on the parameters that
define the family $\F$. We will give some more insight on the
construction of the inducing schemes of \cite{BLS} so that the role
of the constants involved becomes clearer. In what follows we will
use the notation of \cite{BLS} although it may differ from the one we
use in the rest of the paper.

Bruin, Luzzatto and van Strien build a Markov map
 $f^R:\Omega_0\to\Omega_0$ on a small neighbourhood of one of the
 critical points. The key idea is that outside a neighbourhood
 $\Delta$ of the critical points we have hyperbolic behaviour which
 leads to the exponential growth of the derivative. On the other
 hand, when we enter $\Delta$, which happens frequently, we
 have a serious setback on hyperbolicity since the derivative takes
 values very close to $0$. However, because inside $\Delta$ points
 are very close to a critical point they initiate a bind to it and
 regain hyperbolicity on account of the derivative growth
 experienced
 by the critical orbits which we have assumed. Once this binding
  ends, the losses are fully recovered and the derivative then grows
 exponentially fast until we enter $\Delta$ again and the cycle
 repeats itself. This means that the amount of time spent before an
 interval from $\Omega_0$ becomes large enough to cover the whole
 $\Omega_0,$ reflects the growth rate of the derivative along
 critical
 orbits. The construction of the full return map is made in three
 major steps.

 The first step is carried out in \cite[Section~2]{BLS}, where the
 binding argument is described which allows inducing to small
 scales.
 In this section, the following constants are introduced:
  $\kappa$ the \emph{bounded backward contraction constant} (see
  \cite[Lemma~9]{BTeqnat}),
  which depends only on the parameters that define the family $\F$
  and on
  the number of critical points; $K_0$ a Koebe distortion constant
  that
  turns out to be $\leq 16$; $\delta$ which establishes the size of
  the
   critical neighbourhood $\Delta$ and depends on $\kappa$, $K_0$,
   and on the parameters that define the family $\F$. This means
   that
   these constants can be chosen uniformly within the family. Some
   expansion estimates are also derived in this section. In
   \cite[Lemma~2.4]{BLS} the constants $C_\delta$ and
   $\lambda_\delta$, that essentially give the hyperbolicity outside
   the critical region, can be chosen uniformly inside a
   neighbourhood of a fixed $f_0\in\F$. The crucial estimate that
   gives the recovery of
   hyperbolicity after the binding period is
   established in \cite[Lemma~2.5]{BLS}. These estimates depend only
   on the
   parameters that define the family $\F$.

The second step, the most influential step for the tail rate
estimates, is done in \cite[Section~3]{BLS}. It consists in inducing
to large scales which is stated in Proposition~3.1. Essentially, it
is proved that there exists $\delta'>0$ such that for all
$\delta''>0$, any interval $J$ of length at least $\delta''>0$, can
be partitioned in such a way that for every element $\omega$ of the
partition, there is a stopping time $\hat p(\omega)$ such that
$f^{\hat p(\omega)}$ sends $\omega$ diffeomorphically onto an
interval of length at least $\delta'$. Moreover, the tail of this
stopping time function, $m(x\in J:\hat p(x)>n)/m(J)$ decays
accordingly to the derivative expansion on the critical orbits. The
constant $\delta'$ is given by the contraction principle and can be
taken uniformly inside a neighbourhood of any fixed $f_0\in \F$. During
this section, combinatorial estimates are obtained and several
constants that can be chosen uniformly are introduced. The crucial
observation is that the constants obtained for the estimates for
$m(x\in J:\hat p(x)>n)/m(J)$ depend on the parameters that define the
family $\F$, the constants fixed in the previous step and on both
$\delta'$ and $\delta''$. This means that the estimates on the tail
of the time spent to reach large scale are uniform on a neighbourhood
of any $f_0\in\F$.

The third step, in \cite[Section~4]{BLS}, gives the final
construction of the full return map. It starts with
\cite[Lemma~4.2]{BLS} that fixes the base $\Omega_0$ for the inducing scheme that, as we have discussed in the proof of
Claim~\ref{claim:Dstar}, can be chosen uniformly inside a
neighbourhood of any fixed $f_0\in\F$.  Then it is proved that
once an interval achieves large scale, a fixed proportion of it will
make a full return to the base $\Omega_0$ in a finite number of
iterates, which is a property that persists under small perturbations
of $f_0$. This means that once an interval achieves large scale, it
will make a full return exponentially fast. This implies that the
estimates on the tail of the full return map are essentially the ones
obtained in the second step for the time it takes to reach large
scale. The constant $\delta''$ is fixed and its value depends on
$\delta'$ and on the size of $\Omega_0$. All the other constants that appear turn out to depend  on the parameters that define the family
$\F$ and on previous constants, which means that they can all be
chosen uniformly inside a neighbourhood of any fixed $f_0\in\F$.

Note that in \cite[Lemma 9]{BTeqnat} the condition that all the critical
points had to have the same critical order, which had been required
in \cite{BLS}, was removed.
\end{proof}

As a consequence of this lemma, for a given $f_0\in\F$ we can choose
$\kappa=\kappa_{f_0}:\N\to [0,1]$ to be the function so that
$\mu_{F_n}\{\tau_n>s_0\}\le \kappa(s_0)$ for $n$ large enough, and
$\kappa(s_0)\to 0$ as $s_0\to \infty$.

We next make conditions on our inducing schemes, so that only some of
those in Proposition~\ref{prop:convergent ind} will be appropriate
choices. We select our inducing schemes so that the boundary of any
1-cylinder is accumulated by other 1-cylinders. In particular so that
the boundary of a 1-cylinder with a small inducing time is
accumulated by 1-cylinders with larger and larger inducing times.

Since $\x_{n,k}\in \P_{n,k}$ are cylinder sets, $f_n^j(\bd
\x_{n,k})\cap \stackrel{\circ}{\x}_{n,k} =  \es$ for all $1\le j\le
k$.  However we can choose $\x_{0,k}^i\in \P_{0,k}$ so that
$f_0^j(\bd \x_{0,k}^i)\cap \bd \x_{0,k}^i = \es$ for all $1\le j\le
k$ also.  Then as in Lemma~\ref{lem:trans conv}, for each $n$ large
enough, there are corresponding cylinders $\x_{n,k}^i$ with
$f_n^j(\bd \x_{n,k}^i)\cap \bd \x_{n,k}^i = \es$ for all $1\le j\le
k$ also. It is easy to show that this property can be satisfied for
our class of maps.  We denote $\c_{n,0}$ to be the cylinder
$\x_{n,k}^i$, which is fixed for the rest of this paper.  The maps
$F_n=F_{\x_{n,k}^i}$ are defined in Proposition~\ref{prop:convergent
ind}.  Recall that we set $\P_{n,0} := \{\c_{n,0}\}$, and define
$\P_{n,k}$ to be the set of $k$-cylinders for the inducing scheme
$F_n$.  This construction means that $\c_{n,1}^{i_1}\cap
\c_{n,1}^{i_2} = \es$ for all $i_1 \neq i_2$ for all large $n$.  We
exploit this property in Remark~\ref{rem:accum}. We may assume that
this property actually holds for all $n$.

Let $\tau_{n,k}^i$ be the $k$th inducing time on a  cylinder
$\c_{n,k}^i$, i.e. $f^{\tau_{n,k}^i}(\c_{n,k}^i)=\c_{n,0}$. For
brevity we will write $\tau_{n,1}^i=\tau_n^i$.  For use later, note that our potentials $\Psi_{n,k}$ can be written as $\Psi_{n,k}(x)=\Phi_{n,k}(x)-P(\phi_{f_n,t})\tau_{n,k}(x)$ where $\Phi_{n,k}(x):=-\log|DF_n^k(x)|$.

Any element of $\c_{0,k}^i\in \P_{0,k}$ is of the form $\c_{0,k}^i=[a_0,a_1]$ where there is some $p_0, p_1\in\{0,1,\ldots\}$ \st $f_0^{p_i}(a_i)\in \crit_0$ for $i=0,1$. As in the proof Lemma~\ref{lem:trans conv}, for $n$ large enough, depending on $p:=\max\{p_1, p_2\}$, there exists an order preserving bijection $h_{n,p}:\cup_{j=0}^p f_n^{-j}(\crit_n) \to \cup_{j=0}^p f_0^{-j}(\crit_0)$.  Hence there are corresponding points $a_i^n:=h_{n,p}^{-1}(a_i)$, $i=0,1$.  By the proof of Proposition~\ref{prop:convergent ind}, there exists $N\in \N$ \st  for all $n\ge N$, $[a_0^n,a_1^n]$ is a member $\P_{n,k}$, which we can label $\c_{n,k}^i$.  We say that for
$n\ge N$, $\c_{0,k}^i$ is \emph{matched}; \label{p:match} or similarly that
$\c_{n,k}^i$ is \emph{matched}.  In this case, $\c_{n,k}^i \to
\c_{0,k}^i$ as $n \to \infty$.

\begin{remark}
\label{rem:accum} Given $i\ge 1$, for all $M\ge 1$ there  exists
$\eta>0$ and $N\ge 1$ so that for all $n\ge N$,  $(1+\eta)\c_{n,1}^i
\sm \c_{n,1}^i$  only intersects 1-cylinders with $\tau_n>M$. To show
this, we start by choosing $N$ so large that
$\{\c_{n,1}^j:\tau_n^j\le M\}$ are matched for all $n\ge N$. Now let
$\eta:=\frac{1}{2}\min_{j\neq i,\ \tau_0^j\le M} d(\c_{0,1}^i,
\c_{0,1}^j)$.    By the setup, $\eta>0$.  Now we may increase $N$ so
that $n\ge N$ implies $\c_{n,1}^j \cap \left(1+\frac\eta
2\right)\c_{0,1}^j = \c_{n,1}^j$ for all $j$ with $\tau_0^j\le M$.
This means that $\eta$ has the property required.
\end{remark}

\begin{lemma}
\label{lem:matched}For all $\eps>0$ there exists $i_0 \ge 1$ and $N
\ge 1$ \st $\c_{0,1}^i$ is matched for all $1\le i\le i_0$ for all $n
\ge N$, and furthermore $n \ge N$ implies $\mu_{F_n}\left(\bigcup_{i
>i_0}\c_{n,1}^i\right)<\eps$.
\end{lemma}

\begin{proof}
Let $s_0$ be so that $\kappa(s_0)<\eps$.  So $s_0$ depends only on
$\eps$ and $f_0$ as in Lemma~\ref{lem:uniform decay}. We choose $i_0$
so that $\tau_0^i > s_0$ for all $i> i_0$. Similarly to
Remark~\ref{rem:accum}, we can choose $N$ so large that $\c_{n,1}^i$
are matched for all $1\le i\le i_0$ and that $\tau_n^i >s_0$ for all
$i> i_0$ and all $n\ge N$.  It then follows that
$\mu_{F_n}\left(\bigcup_{i >i_0}\c_{n,1}^i\right)<\eps$ as required.
\end{proof}

In the following lemmas we repeatedly use the conformal property of
$m_{F_n}$ for $n=0,1,2\ldots$. This allows us to compare behaviour at
small scales with that at large scale.

\begin{lemma}
For all $\eps>0$ for all $i_0 \ge 1$ there exists $\eta>0$, \st for
all $k\ge 1$, any  $\c_{0,k}^j\in \P_{0,k}$ with
$F_0^{k-1}(\c_{0,k}^j) = \c_{0,1}^i$ and $1\le i\le i_0$ has
$$\frac{m_{F_0}\left((1+\eta)\c_{0,k}^j\right)}{m_{F_0}(\c_{0,k}^j)}
\le B_0\left(1+\frac\eps4\right) \hbox{ and } \frac{m_{F_0}\left(
\left(\frac1{1+\eta} \right) \c_{0,k}^j\right)}{m_{F_0}(\c_{0,k}^j)}
\ge \frac1{B_0\left(1+\frac\eps4\right)}.$$ \label{lem:F0 fringe}
\end{lemma}

\begin{proof}
  Let $s_1\ge 1$ be \st
  $$\kappa(s_1) \le \frac\eps 4 \left(\min_{1\le i
  \le i_0}m_{F_0}(\c_{0,1}^i)\right).$$For the upper bound,
  let $\eta'>0$ be such that the set $\bigcup_{1
\le i\le i_0} (1+B_0\eta')\c_{0,1}^i \sm\c_{0,1}^i$ contains only
cylinders $\c_{0,1}^i$ with $\tau_0^i \ge s_1$ and is contained in
$\c_{0,0}$.  Then $m_{F_0}((1+B_0\eta') \c_{0,1}^i) \le
\left(1+\frac\eps4\right)m_{F_0}(\c_{0,1}^i)$ for $1\le i\le i_0$.

For $k>1$, we use distortion and conformality to reduce the problem
to the 1-cylinders' case just considered. Assume for $k\ge 1$ that
$F_0^{k-1}(\c_{0,k}^j) = \c_{0,1}^i$. Since $\left(1+\eta'
\right)\c_{0,k}^j$ is in the same $k-1$-cylinder as
$\c_{0,k}^j$, bounded distortion implies that
$$F_0^{k-1}\left(\left(1+\eta'\right) \c_{0,k}^j\right) \subset (1+B_0\eta')\c_{0,1}^i.$$
Using the conformal property of $m_{F_0}$ and bounded distortion we
have
$$\frac{m_{F_0}\left(\left(1+B_0\eta'\right)\c_{0,1}^i\right)
}{m_{F_0}\left(\c_{0,1}^i\right)} \ge \frac{\int_{\left(1+\eta'
\right)\c_{0,k}^j}
e^{-\Psi_{0,k-1}}~dm_{F_0}}{\int_{\c_{0,k}^j}e^{-\Psi_{0,k-1}}~dm_{F_0}}
\ge \frac1{B_0}\left( \frac{m_{F_0}\left(\left(1+\eta'\right)
\c_{0,k}^j \right)}{m_{F_0}(\c_{0,k}^j)}\right).$$ Hence, by the
choice of $\eta'$ above, we have
\[
\frac{m_{F_0}\left(\left(1+\eta'\right) \c_{0,k}^j
\right)}{m_{F_0}(\c_{0,k}^j)}\le
B_0\left(1+\frac\varepsilon4\right).
\]

For the lower bound, let $s_2\ge 1$ be such that
$\kappa(s_2)<\frac\eps8$.  Then we choose $0<\eta\le \eta'$ so that
the set $\c_{0,0} \sm\frac{\c_{0,0}}{1+B_0\eta}$ only contains
1-cylinders $\c_{0,1}^i$ with $\tau_0^i\ge s_2$.  This implies
$m_{F_0}\left(\left(\frac1{1+B_0\eta}\right)\c_{0,0}\right) \ge
1-\frac\eps8 >\frac1{\left(1+\frac\eps4\right)}$.

For $k>1$ we use the a distortion argument similar to the one above.
Bounded distortion implies that
$$F_0^{k}\left(\left(\frac1{1+\eta} \right) \c_{0,k}^j\right)
\supset\left(\frac1{1+B_0\eta}\right)\c_{0,0}.
   $$
Using the conformal property of $m_{F_0}$ and bounded distortion we
have
$$\frac{m_{F_0}\left(\left(\frac1{1+B_0\eta}\right)\c_{0,0}\right)
}{m_{F_0}\left(\c_{0,0}\right)} \le \frac{\int_{\left(\frac1{1+\eta}
\right) \c_{0,k}^j}
e^{-\Psi_{0,k}}~dm_{F_0}}{\int_{\c_{0,k}^j}e^{-\Psi_{0,k}}~dm_{F_0}}
\le B_0\left( \frac{m_{F_0}\left(\left(\frac1{1+\eta} \right)
\c_{0,k}^j \right)}{m_{F_0}(\c_{0,k}^j)}\right).$$ Hence, by the
choice of $\eta$ above, we have
\[
\frac{m_{F_0}\left(\left(\frac1{1+\eta} \right) \c_{0,k}^j
\right)}{m_{F_0}(\c_{0,k}^j)}\ge
\frac1{B_0\left(1+\frac\eps4\right)}.
\]
\end{proof}

Notice that the above proof can be used to show that
for the 1-cylinders considered above,
$$\frac{m_{F_0}\left(\c_{0,1}^j \sm \left(\frac1{1+\eta} \right)
\c_{0,1}^j\right)}{m_{F_0}(\c_{0,1}^j)} \le
\frac{B_0}{1+\frac\eps4}.$$

\begin{proposition}
For all $\eps>0$, $\lambda\in (0,1)$, $k_0 \ge 1$ and sequences
$(i_1,\ldots, i_{k_0})\in \N^{k_0}$ there exists $N_0\ge 1$ \st for
all $n\ge N_0$, $1\le k\le k_0$ and $1\le i\le i_k$, we have
$$\frac1{B_0^2(1+\eps)} \le
\frac{m_{F_n}(\c_{0,k}^i)}{e^{\Psi_{0,k}(x)}}
 \le B_0^2(1+\eps)$$ for all $x\in \lambda \c_{0,k}^j$.
\label{prop:Fn fringe}
\end{proposition}

\begin{proof}
The following claim is left to the reader.

\begin{claim}
For all $\eps>0$, $k_0 \ge 1$ and sequences $(i_1,\ldots, i_{k_0})\in
\N^{k_0}$ there exists $N_0\ge 1$ \st for all $n\ge N_0$, $1\le k\le
k_0$ and $1\le i\le i_k$,  $\c_{n,k}^i$ is matched.  Moreover, for
these cylinders, each set $F_n^{k-1}(\c_{n,k}^i)$ is matched.
\label{claim:matching}
\end{claim}

We next make the following claim.
\begin{claim} For all $\eps>0$, $k_0 \ge 1$ and
sequences $(i_1,\ldots, i_{k_0})\in \N^{k_0}$ there exist $\eta>0$ and $N_1\ge
N_0$ \st for all $n\ge N_1$, $1\le k\le k_0$ and $1\le i\le i_k$,
$$\frac{m_{F_n}((1+\eta)\c_{n,k}^i)}{m_{F_n}(\c_{n,k}^i)} \le B_0
\left(1+\frac\eps4\right) \hbox{ and } \frac{m_{F_n}\left(\left(
\frac1{1+\eta}\right)\c_{n,k}^i\right)}{m_{F_n}(\c_{n,k}^i)} \ge
\frac1{B_0\left(1+\frac\eps4\right)}.$$\label{claim:fringes}\end{claim}
\begin{proof} The proof of the claim is the same as for Lemma~
\ref{lem:F0 fringe} except that we need to take $F_n$ sufficiently
close to $F_0$ so that the cylinders $\c_{n,k}^i$ considered in
Lemma~\ref{lem:F0 fringe} have almost exactly the same properties as
 those $\c_{0,k}^i$ considered here.
\end{proof}

A simple consequence of these claims is that for all $\eps>0$,
$k_0 \ge 1$ and sequences $(i_1,\ldots,
i_{k_0})\in \N^{k_0}$ there exists $N_2\ge N_1$ \st for all $n\ge
N_2$, $1\le k\le k_0$ and $1\le i\le i_k$, $\c_{n,k}^i\in \P_{n,k}$
is matched and
$$\frac1{B_0\left(1+\frac\eps4\right)} \le
\frac{m_{F_n}(\c_{0,k}^i)} {m_{F_n}(\c_{n,k}^i)} \le
B_0\left(1+\frac\eps4\right).$$ Here we choose $N_2 \ge N_1$ so that
$\c_{0,k}^j \subset (1+\eta)\c_{n,k}^j$ and $\c_{n,k}^j \subset
(1+\eta)\c_{0,k}^j$ for all $\c_{n,k}^j$ as in the statement of the
proposition.

The Gibbs property for $m_{F_n}$, which follows directly from
conformality, means that $m_{F_n}(\c_{n,k}^i) =B_0^\pm
e^{\Psi_{n,k}(x)}$ for all $x\in \c_{n,k}^i$.  Now we can take $N_2$
so large that $$\frac1{\left(1+\frac\eps 4\right)} \le
e^{\Psi_{n,k}(x)-\Psi_{0,k}(x)}=
e^{\Phi_{F_n,k}(x)-\Phi_{F_0,k}(x)+\left(P(\phi_{f_n,t})-P(\phi_{f_0,t})\right)
\tau_{0,k}(x)}\le 1+\frac\eps4$$ for all $x\in \c_{n,k}^i\cap
\c_{0,k}^i$ for the cylinders $\c_{n,k}^i$ under consideration. This
follows since $\Phi_{F_n,k}(x) \to \Phi_{F_0,k}(x)$ as $n\to \infty$,
and by Proposition~\ref{prop:press cts}, $P(\phi_{f_n,t})\to
P(\phi_{f_0,t})$ as $n\to \infty$.  To complete the proof of the
proposition, we possibly increase $N_2$ again to ensure that
$\c_{n,k}^i\cap \c_{0,k}^i \subset \lambda \c_{0,k}^i$ for all the
cylinders we consider.
\end{proof}
Combining Lemma~\ref{lem:H} and Proposition~\ref{prop:Fn fringe} we
have that
 $\mu_{F_\infty}$ must have the Gibbs property with uniform
 constant. That is:
\begin{corollary}
\label{cor:gibbs-property-of-mu} For all $k$ and all $\c_{0, k}\in
\P_{0,k}$,
$$\frac1{H_{\F}B_0^2(1+\eps)} \le \frac{\mu_{F_\infty}(\c_{0,k}^i)}
{e^{\Psi_{0,k}(x)}} \le H_{\F}B_0^2(1+\eps),$$ for all $x\in
\c_{0,k}^i$.
\end{corollary}

We will need the following lemma later.

\begin{lemma}\label{lem:meas-prox-cylin}
For all $\eps>0$ and  $i_0\ge 1$ there exists $N \ge 1$ \st
$n\ge N$ implies
$$\mu_{F_n}\left(\bigcup_{i=0}^{i_0}\left( \c_{n,1}^i \triangle
 \c_{0,1}^i\right)\right) \le \eps.$$
\label{lem:cyls coincide}
\end{lemma}

\begin{proof}
Combining the arguments in the proof of  Lemma~\ref{lem:F0 fringe},
the paragraph following it and Claim~\ref{claim:fringes} in the proof
of Proposition~\ref{prop:Fn fringe} we have $\eta>0$, $i_0\ge 1$ and
$N' \ge 1$ \st for $n \ge N'$,
$$m_{F_n}\left((1+\eta)\c_{n,k}^i\sm\c_{n,k}^i\right), \
m_{F_n}\left(\c_{n,k}^i\sm\frac{\c_{n,k}^i}{(1+\eta)}\right)<\frac\eps{i_0H_{\F}}$$
for all  $1 \le i\le i_0$.  Recall that $H_{\F}$ is the constant from
Lemma~\ref{lem:H}. Moreover, there exists $N \ge N'$ \st $n \ge N$
implies $$\c_{n,1}^i \triangle \c_{0,1}^i \subset
(1+\eta)\c_{n,k}^i\sm \frac{\c_{n,k}^i}{(1+\eta)}$$ for all $1\le
i\le i_0$.  Therefore, $n\ge N$ implies
$$m_{F_n}\left(\bigcup_{i=0}^{i_0}\left( \c_{n,1}^i \triangle
\c_{0,1}^i\right)\right) \le \frac\eps{H_{\F}}.$$  The lemma follows
from Lemma~\ref{lem:H}, substituting $\mu_{F_n} (=\rho_{F_n}m_{F_n})$
for $m_{F_n}$ in the above equation.
\end{proof}

We finish this section by proving Proposition~\ref{prop:press cts},
which was essential in the proof of Proposition~\ref{prop:Fn
fringe}.

\begin{proof}[Proof of Proposition~\ref{prop:press cts}]
Observe that for $t=1$ there is nothing to prove since
$P(\phi_{f,1})=0$ for all $f\in \F$.

Let $\eps>0$.  We fix $t< 1$ as in the statement of the proposition,
since the proof for $t>1$ (which we need only consider when
\eqref{eq:CE} holds) follows similarly.  For ease of notation, we let
$P_n:=P(\phi_{f_n,t})$. For $S\in \R$, we define
$\psi_n^S:=\phi_{f_n,t}-S$.  Likewise, for $k\ge 1$ the corresponding
induced potentials are  $\Psi_{n,k}^S:= \Phi_{n,k}-S\tau_{n,k}$.

We choose $\c_{0,1}^i$ and $n_1\ge 1$ so that for all $n\ge n_1$,
this cylinder is matched.  Recall that we can write
$$P_G\left(\Psi_n^{P_n}\right) = P_G\left(\Psi_{n,1}^{P_n},
\c_{n,1}^i\right)= \lim_{k\to \infty}\frac1k \log
Z_k\left(\Psi_{n,1}^{P_n}, \c_{n,1}^i\right)=0,$$
where $Z_k(\Psi_{n,1}^{P_n}, \c_{n,1}^i)= \sum_{x\in \c_{n,1}^i, \
F_n^k(x)=x}e^{\Psi_{n,k}^{P_n}(x)}$.  The idea of this proof is to
use the fact that $S=P_n$ is the unique value so that
$P_G(\Psi_{n,1}^S, \c_{n,1}^i)=0$.  We show that since $\frac1k\log
Z_k(\Psi_0^{P_n}, \c_{0,1}^i)$ and $\frac1k\log Z_k(\Psi_n^{P_n},
\c_{n,1}^i)$ are close to each other for all large $n$, with the
latter value converging to 0 as $k\to \infty$, then $P_0$ and $P_n$
must also be close.

We first show that the convergence of $\frac1k\log Z_k(\Psi_n^{P_n},
\c_{n,1}^i)$ to 0 as $k\to \infty$ is essentially uniform in $n$. The
Gibbs property together with Lemma~\ref{lem:H} imply that for a
 cylinder $\c_{n,k}\in \P_{n,k}$ we have
 $\mu_{\Psi_n}^{P_n}(\c_{n,k})=
  \H_{\F}^{\pm}e^{\Psi_{n,k}(x)}$ for any $x\in \c_{n,k}$,  Since
  each $k$-cylinder
contains a unique $k$-periodic point, it follows that
$$Z_k\left(\Psi_{n,1}^{P_n}, \c_{n,1}^i\right) = H_{\F}^\pm
\sum_{\c_{n,k}\in \P_{n,k},\\ \c_{n,k}\subset \c_{n,1}}
\mu_{\Psi_n^{P_n}}(\c_{n,k}) = H_{\F}^\pm
\mu_{\Psi_n^{P_n}}(\c_{n,1}^i).$$  We now show that
$\mu_{\Psi_n^{P_n}}(\c_{n,1})$ is uniformly bounded above and below
for all large $n$. Again using the Gibbs property we have
$$\frac{1}{H_{\F}} \le
\frac{\mu_{\Psi_n^{P_n}}(\c_{n,1}^i)}{e^{\Phi_n(x)-P_n\tau_n(x)}} \le
H_{\F}$$ for any $x\in \c_{n,1}^i$.  Since there exists a uniform
$K_{\F}\in \R$ so that for all $n\ge 0$, $P_n\in [0,K_{\F}]$;
$\c_{n,1}^i$ is matched (recall the definition on \pageref{p:match}), for all $n\ge n_1$; and $\Phi_n$ converges to
$\Phi_0$ on $\c_{n,1}^i\cap \c_{0,1}^i$, there exists $n_2\ge n_1$ so
that $n=0$ or $n\ge n_2$ implies
$$\frac{e^{-K_{\F}\tau_0(x)}}{H_{\F}(1+\eps)} \le
\frac{\mu_{\Psi_n^{P_n}}(\c_{n,1}^i)}{e^{\Phi_0(x)}} \le
H_{\F}(1+\eps)$$ for any $x\in \c_{0,1}^i$. Combining the above
computations we get, for $n=0$ or $n\ge n_2$,
\begin{equation}\frac{e^{\Phi_0(x)-K_{\F}\tau_0(x)}}{H_{\F}^2(1+\eps)}\le
Z_k \left(\Psi_{n,1}^{P_n}, \c_{n,1}^i\right)\le
H_{\F}^2(1+\eps)e^{\Phi_0(x)}\label{eq:decay of Z_k}\end{equation}
for any $x\in \c_{0,1}^i$.

Therefore for any $\delta>0$ there exists a uniform $k=k(\delta)\ge
1$ so that $$\left|\frac1k\log Z_k\left(\Psi_{n,1}^{P_n},
\c_{n,1}^i\right) \right|<\delta$$ for $n=0$ and all $n\ge n_2$. We
fix $\delta=\frac\eps 5$.

In order to be able to prove this proposition using only a finite
amount of information, we define
$$Z_k\left(\Psi_{n,1}^{P_n}, \c_{n,1}^i,N\right):=
\sum_{x\in \c_{n,1}^i, \ F^k(x)=x,\ \tau_{n,k}(x)\le N}e^{\Psi_{n,k}^
{P_n}(x)}.$$ By the Gibbs property, the quantity
$Z_k(\Psi_{n,1}^{P_n}, \c_{n,1}^i)-Z_k(\Psi_{n,1}^{P_n},
\c_{n,1}^i,N)$ is, up to a constant, the measure of the `tail set for
$F_n^k$', i.e. $\mu_{\Psi_n^{P_n}}\{\tau_{n,k}>N\}$. It is easy to
see that estimating this last quantity can be reduced to estimating
$\mu_{\Psi_n^{P_n}}\{\tau_{n}>N\}$ (see for example
\cite[Lemma~9.7]{ACF}). Therefore the following claim follows for the
same reasons as in Lemma~\ref{lem:uniform decay}. (Recall here that
we are not dealing with the case $t=1$, so even if we only assume
that \eqref{eq:poly} holds on $\F$, since we are considering $t<1$,
we have exponential tails.) In contrast to the notation in the rest
of this proof, here, for clarity, we emphasise the role of $\F$ and
$t$.

\begin{claim}
There exist $C_{\F,t}, \alpha_{\F,t}>0$ so that for $n=0$ or $n\ge
n_2$,  and for all $k\ge 0$, $$Z_k\left(\Psi_{n,1}^{P_n},
\c_{n,1}^i\right)-Z_k\left(\Psi_{n,1}^{P_n}, \c_{n,1}^i,N\right) \le
C_{\F,t}e^{-\alpha_{\F,t} N}.$$
\end{claim}

Adding \eqref{eq:decay of Z_k} and the claim together, we can fix
$N=N(\delta,n_1)\ge 1$ so that for all $n\ge n_2$, \begin{equation}
e^{-3\delta k}\le \frac{ Z_k\left(\Psi_{n,1}^{P_n},
\c_{n,1}^i,N\right)}{Z_k\left(\Psi_{0,1}^{P_0},
\c_{0,1}^i,N\right)}\le e^{3\delta k}.\label{eq:trunk
Z_k}\end{equation}

Now since $\Psi_{n,k}^{S}$ on the set $\{\tau_{0,k}\le N\}$
essentially converges to $\Psi_{0,k}^{S}$ as $n\to \infty$, there
exists $n_3\ge n_2$ so that $n\ge n_3$ implies $e^{-k\delta} \le
\frac{Z_k\left(\Psi_{0,1}^{P_n},
\c_{0,1}^i,N\right)}{Z_k\left(\Psi_{n,1}^{P_n},
\c_{n,1}^i,N\right)}\le e^{k\delta}$, and hence
$$e^{-4\delta k}\le \frac{ Z_k\left(\Psi_{0,1}^{P_n},
\c_{0,1}^i,N\right)} {Z_k\left(\Psi_{0,1}^{P_0},
\c_{0,1}^i,N\right)}\le e^{4\delta k}.$$ But $\tau_{0,k} \ge k$, so
by the definition of $Z_k\left(\Psi_{0,1}^{S}, \c_{0,1}^i,N\right)$,
the above inequalities imply that $|P_n-P_0|$ must be less than or
equal to $4\delta<\eps$ as required.
\end{proof}

\section{Invariance of the weak$^*$ limit}
\label{sec:invariance-weak-limit}

We may assume, as in the beginning of
Section~\ref{sec:gibbs-property}, that $\mu_{F_\infty}$ is the
weak$^*$ limit of the sequence $(\mu_{F_n})_n$. In the
previous section we saw that $\mu_{F_\infty}$ is Gibbs. The purpose
of this section is to show that $\mu_{F_\infty}$ is $F_0$-invariant.
Before that, we prove the following technical lemma that will be
useful in the remaining arguments.
\begin{lemma}
\label{lem:indicator-convergence} For all $i\in\N$ and every
continuous $g:C_{0,1}^i\to\R$ we have
\[
\int g.\I_{C_{0,1}^i}~d\mu_{F_n}\rightarrow\int
g.\I_{C_{0,1}^i}~d\mu_{F_\infty}.
\]
\end{lemma}
\begin{proof} We can extend $g$ continuously to $\partial
C_{0,1}^i$,
and for every $x\in I\setminus \overline{C_{0,1}^i}$, define
$\text{b}^i(x)$ as the point of $\partial C_{0,1}^i$ closest to $x$.

Observing that $g=g^+-g^-$, where $g^+(x)=\max\{0,g(x)\}\ge0$ and
$g^-(x)=\max\{0,-g(x)\}\ge0$, we may assume without loss of
generality that $g\ge0$. Also, since, by
Corollary~\ref{cor:gibbs-property-of-mu}, $\mu_{F_\infty}$ is a Gibbs
measure, we have $\mu_{F_\infty}(\partial C_{0,1}^i)=0$, which
implies that $\int_{\overline{C_{0,1}^i}\setminus \partial
C_{0,1}^i}g~d\mu_{F_\infty}=\int_{\overline{C_{0,1}^i}}gd\mu_{F_\infty}=
\int_{C_{0,1}^i}g~d\mu_{F_\infty}$.

Let $U_k=\{x\in I: \mbox{dist}(x,\overline{C_{0,1}^i})<1/k\}$.
Clearly $U_k$ is an open neighbourhood of $\overline{C_{0,1}^i}$ and
by the regularity of $\mu_{F_\infty}$ it follows that
$\mu_{F_\infty}(U_k\setminus \overline{C_{0,1}^i})=\epsilon(k)\to 0$
as $k\to\infty$. Define $h:I\to\R$ as
\[
h(x)=\begin{cases}
  0&\text{if $x\notin U_k$}\\
  g(\text{b}^i(x))\frac{d(x,I\setminus U_k)}{d(x,I\setminus U_k)+
  d(x,\overline{C_{0,1}^i})}&
  \text{if $x\in U_k\setminus\overline{C_{0,1}^i}$}\\
  g(x)& \text{if $x\in \overline{C_{0,1}^i}$}
\end{cases}.
\]
Notice that $h$ is continuous and, for every $x\in I$, we have
$g(x)\I_{C_{0,1}^i}(x)\le h(x)\le \max_{x\in
\overline{C_{0,1}^i}}g(x)$ and $h(x)-g(x)\I_{C_{0,1}^i}(x)>0$ only if
$x\in U_k\setminus\overline{C_{0,1}^i}$. Consequently, using the
weak$^*$ convergence of $\mu_{F_n}$ to $\mu_{F_\infty}$, it follows
\[
\int g\I_{C_{0,1}^i}d\mu_{F_n}\le \int h
~d\mu_{F_n}\xrightarrow[n\to\infty]{}\int h~d\mu_{F_\infty}\le \int
g\I_{\overline{C_{0,1}^i}}~d\mu_{F_\infty}+\epsilon(k)\max_{x\in
\overline{C_{0,1}^i}}g(x).
\]
Letting $k\to\infty$ we get $\int g\I_{C_{0,1}^i}d\mu_{F_n}\le \int
g\I_{C_{0,1}^i}d\mu_{F_\infty}$. The opposite inequality follows
similarly.
\end{proof}

\begin{lemma}
\label{lem:invariance-of-mu} $\mu_{F_\infty}$ is $F_0$-invariant.
\end{lemma}
\begin{proof}
The $F_0$-invariance of $\mu_{F_\infty}$ is equivalent to
  $$
\int \varphi\circ F_0~d\mu_{F_\infty}= \int\varphi~d\mu_{F_\infty}
  $$
for every continuous $\varphi\colon I\to\mathbb R$. Given any
$\varphi\colon I\rightarrow\mathbb R$ continuous we have by
hypothesis
 $$
\int \varphi~d\mu_{F_n}\rightarrow \int \varphi ~d\mu_{F_\infty}\quad
\mbox{as}\quad n\rightarrow\infty.
 $$ On
the other hand, since $\mu_{F_n}$ is an $F_{n}$-invariant probability
measure, we have
 $$
 \int \varphi~d\mu_{F_n}=\int
(\varphi\circ F_{n})~d\mu_{F_n}\quad\mbox{for every }n\ge0.
 $$
 So, it suffices to prove that
 \begin{equation}
 \int (\varphi\circ F_{n})~d\mu_{F_n}
 \rightarrow \int (\varphi\circ F_0)~d\mu_{F_\infty}
 \quad\mbox{as}\quad n\rightarrow\infty.
 \end{equation}
 We have
 \begin{eqnarray*}
 \lefteqn{
 \left|\int (\varphi\circ F_{n})~d\mu_{F_n} -\int (\varphi\circ F_0)
  ~d\mu_{F_\infty}\right| \le }\\
& & \left|\int (\varphi\circ F_{n})~d\mu_{F_n}
 - \int (\varphi\circ F_0)~d\mu_{F_n}\right|
+\left|\int (\varphi\circ F_0)~d\mu_{F_n}
 - \int (\varphi\circ F_0)~d\mu_{F_\infty}\right|.
 \end{eqnarray*}
Observing that $\varphi\circ F_0$ is continuous on each $C_{0,1}^i$,
we easily deduce from Lemma~\ref{lem:indicator-convergence} and
Lemma~\ref{lem:matched} that the second term in the sum above is
close to zero for large $n$.

The only thing we are left to prove is that the  first term in the
sum above converges to 0 when $n$ tends to $\infty$.  That term is
bounded above by
\begin{equation}
\label{eq:term-1}
 \int \big|\varphi\circ F_{n}-\varphi\circ F_0\big|~d\mu_{F_n}.
 \end{equation}
Take any $\varepsilon>0$. Using Lemma~\ref{lem:uniform decay}, take
$N\ge 1$ \st
 $$
 \sum_{\tau_n^i>N}\mu_{F_n}(C_{n,1}^i)<\varepsilon.
 $$
We write the integral in \eqref{eq:term-1}  as
 \begin{equation}\label{eq:somatau}
  \sum_{\tau_n^i>N}\int_{C_{n,1}^i} \big|\varphi\circ
  F_{n}-\varphi\circ F_0\big|
 ~d\mu_{F_n}+\sum_{\tau_n^i\le N}\int_{C_{n,1}^i} \big|\varphi\circ
 F_{n}-
 \varphi\circ F_0\big|~d\mu_{F_n}.
 \end{equation}
 The first sum in \eqref{eq:somatau} is bounded by
 $2\varepsilon\|\varphi\|_\infty$. Let us now estimate
 the second sum in \eqref{eq:somatau}.

Using Lemma~\ref{lem:matched}, we take $n_1$ sufficiently large so
that for all $n\ge n_1$ and every cylinder $C_{n,1}^i$ with
$\tau_n^i\le N$ there is a matching cylinder $C_{0,1}^i$ with
$\tau_n^i=\tau_0^i$. Moreover, using Lemma~\ref{lem:meas-prox-cylin},
we may assume that $n_1$ is large enough so that $n\ge n_1$ implies
\[
\sum_{\tau_n^i\le N}\mu_{F_n}(C_{n,1}^i\triangle
C_{0,1}^i)<\varepsilon.
\]
For every $i$ such that $\tau_n^i\le N$ we have
\begin{align*}
  \int_{C_{n,1}^i} \big|\varphi\circ F_{n}-\varphi\circ
  F_0\big|~d\mu_{F_n}
  \le & \int_{C_{n,1}^i\cap C_{0,1}^i} \big|\varphi\circ
  f_{n}^{\tau_0^i}  -\varphi \circ f_0^{\tau_0^i}\big|~d\mu_{F_n}\\
  & +\int_{C_{n,1}^i\setminus C_{0,1}^i} \big|\varphi\circ F_{n}-
  \varphi\circ F_0\big|~d\mu_{F_n}.
\end{align*}
Since $f_n\to f_0$ in the $C^k$ topology, there is $n_2\in \mathbb N$
such that for $n\ge n_2$
 $$
 \sum_{\tau_n^i\le N}\int_{C_{n,1}^i\cap C_{0,1}^i}\big|\varphi\circ
 f_{n}^{\tau_n^i}-\varphi\circ
 f_0^{\tau_n^i}\big|~d\mu_{F_n}
 < \varepsilon .
 $$
 On the other hand, for $n\ge n_1$
 $$
 \sum_{\tau_n^i\le N}\int_{C_{n,1}^i\triangle C_{0,1}^i}
 \big|\varphi\circ
 F_{n}-\varphi\circ F_0\big|~d\mu_{F_n}
 \le 2\varepsilon\|\varphi\|_\infty.
 $$
Thus we have for $n\ge \max\{n_1,n_2\}$
 $$
 \int \big|\varphi\circ F_{n}-\varphi\circ F_0\big|~d\mu_{F_n}\le
 \varepsilon\big(4\|\varphi\|_\infty+1\big).
 $$
This proves the result since $\varepsilon>0$ was arbitrary.
\end{proof}

Since $\mu_{F_\infty}$ is an invariant Gibbs measure, uniqueness of
such measures, \cite[Theorem 3.2]{MauUrb}, implies $\mu_{F_\infty}
\equiv \mu_{F_0}$.

\begin{remark}
\label{rem:convergence-whole-seq} Observe that the whole sequence
$\mu_{F_n}$ converges in the weak$^*$ topology to $\mu_{F_0}$. This
is because any subsequence $\left(\mu_{F_{n_i}}\right)_i$ admits a
convergent subsequence $\left(\mu_{F_{n_{i_j}}}\right)_j$, whose
weak$^*$ limit, $\mu_{F_\infty}$, is Gibbs and $F_0$-invariant, by
Corollary~\ref{cor:gibbs-property-of-mu} and
Lemma~\ref{lem:invariance-of-mu}. Hence, by uniqueness,
$\mu_{F_{n_{i_j}}}\to\mu_{F_0}$, in the weak$^*$ topology, which
clearly implies the statement.
\end{remark}

\section{Continuous variation of equilibrium states}
\label{sec:continuity-equilibrium-states}

So far, we managed to prove that if $\|f_n- f_0\|_{C^2}\to 0$, then
the induced Gibbs measures converge in the weak$^*$ topology, ie,
$\mu_{F_n}\to\mu_{F_0}$.  Similarly to \eqref{eq:lift}, we let
$\mu^*_f/\int \tau d\mu_F$ be
the projection of the measure $\mu_F$, i.e.
\begin{equation}
\label{eq:saturation-definition}
\mu^*_f=\sum_{i=1}^\infty\sum_{k=0}^{\tau_{i}-1}
f^k_*\left(\mu_F|C_1^i\right)
\end{equation}
By Lemma~\ref{lem:uniform decay}, the total mass of this measure
$\mu^*_{f_n}(I)$ is uniformly bounded in $n$. Observe that, for some
fixed $t\in U_{\F}$, the unique equilibrium state of $f_n$ for the
potential $-t\log\left|Df_n\right|$ is such that
$\mu_n=\mu^*_{f_n}/\mu^*_{f_n}(I)$, for every $n\ge 0$. Consequently,
the proof of Theorem~\ref{thm:main} will be complete once we prove:

\begin{proposition}
 For every continuous
$g:I\rightarrow I$,
\[
\int g~d\mu^*_{f_n}\xrightarrow[n\rightarrow\infty]{}\int
g~d\mu^*_{f_0}.
\]
\end{proposition}
\begin{proof}
First observe that as $I$ is compact, $g$ is uniformly continuous and
$\|g\|_\infty<\infty$.

Let $\varepsilon>0$ be given. We look for $n_0\in\N$ sufficiently large
so that for every $n>n_0$
\[
\left| \int g~d\mu^*_{f_n}-\int g~d\mu^*_{f_0}\right|<\varepsilon
\]
Recalling  \eqref{eq:saturation-definition} we may write for any
integer $N$
\[
\mu^*_{f_n}=\sum_{\tau_n^i\le
N}\sum_{k=0}^{\tau_n^i-1}(f_n^k)_*(\mu_{F_n}|C_{n,1}^i)+\eta_{f_n}
\mbox{ and } \mu^*_{f_0}=\sum_{\tau_0^i\le
N}\sum_{k=0}^{\tau_0^i-1}(f_0^k)_*(\mu_{F_0}|C_{0,1}^i)+\eta_{f_0}
\] where
$\eta_{f_n}=\sum_{\tau_n^i>
N}\sum_{k=0}^{\tau_n^i-1}(f_n^k)_*(\mu_{F_n}|C_{n,1}^i)$ and
$\eta_{f_0}=\sum_{\tau_0^i>
N}\sum_{k=0}^{\tau_0^i-1}(f_0^k)_*(\mu_{F_0}|C_{0,1}^i)$. Using
Lemma~\ref{lem:uniform decay} we pick $N$ large enough so that $n\ge
N$ implies
\[
\eta_{f_n}(I)+\eta_{f_0}(I)<\varepsilon/2.
\]
Using Lemma~\ref{lem:matched}, we take $n_1$ sufficiently large so
that for all $n\ge n_1$ and every cylinder $C_{n,1}^i$ with
$\tau_n^i\le N$ there is a matching cylinder $C_{0,1}^i$ with
$\tau_n^i=\tau_0^i$. Let $S_N$ denote the number of 1-cylinders such
that $\tau_n^i\le N$.  To complete the proof of the proposition, for
every $i$ such that $\tau_n^i\le N$ and $k<\tau_n^i$, we must find a
sufficiently large $n_2$ so that for every $n\ge n_2$
\[
E:=\left| \int (g\circ f_{n}^k)
 \I_{C_{n,1}^i}~d\mu_{F_n}-\int (g\circ f_{0}^k)
\I_{C_{0,1}^i}~d\mu_{F_0}\right|<\frac{ \varepsilon}{2S_N}.
\]
We split $E$ into $E_1$, $E_2$ and $E_3$ presented in respective
order:
\begin{align*}
E\le&\left|\int \left[(g\circ f_{n}^k)
 -(g\circ f_{0}^k)\right]
\I_{C_{n,1}^i}d\mu_{F_n}\right|\\ &+\left|\int (g\circ
f_{0}^k)\left[
 \I_{C_{n,1}^i}-\I_{C_{0,1}^i}\right]
d\mu_{F_n}\right|\\ &+\left|\int (g\circ f_{0}^k)
 \I_{C_{0,1}^i}
d\mu_{F_n}-\int (g\circ f_{0}^k) \I_{C_{0,1}^i}d\mu_{F_0}\right|.
\end{align*}
Since
\[
E_1\le \int \left|(g\circ f_{n}^k)
 -(g\circ f_{0}^k)\right|~d\mu_{F_n},
\]
we choose $n_2$ large enough so that for every $n>n_2$ we have
$\left|(g\circ f_{n}^k)
 -(g\circ f_{0}^k)\right|\le\frac{ \varepsilon}{6S_N}$ in order to
 obtain $E_1\le\frac{ \varepsilon}{6S_N}$.

Now,
\[
E_2\le \|g\|_\infty\mu_{F_n}(C_{n,1}^i\triangle C_{0,1}^i).
\]
Using Lemma~\ref{lem:meas-prox-cylin}, we take $n_2$ large enough so
that for all $n>n_2$ we have $E_2\le\frac{\varepsilon}{6S_N}$.

Regarding the last term, Lemma~\ref{lem:indicator-convergence} allows
us to conclude that if $n_2$ is sufficiently large then for all
$n>n_2$ we have $E_3\le\frac{ \varepsilon}{6S_N}$.
\end{proof}

\begin{proof}[Proof of Theorem~\ref{thm:strong-stat-stab}]
Alves and Viana, in \cite{AV:2002}, give some abstract conditions for
statistical stability of physical measures in the strong sense, that
is, convergence of densities in the sense of \eqref{eq:cty}.
Essentially, they consider a family $\mathcal U$ of $C^k$ ($k\geq2$)
maps  admitting an inducing scheme. Their main result,
\cite[Therorem~A]{AV:2002}, asserts that if some uniformity
conditions, which they denote by $U_1$, $U_2$ and $U_3$, hold within
the family, then one gets strong statistical stability. Condition $U_1$ requires that cylinders with
finite inducing times are arbitrarily close with respect to the
reference measure, just as we have shown for our inducing schemes in
the proof of Lemma~\ref{lem:meas-prox-cylin}. Condition $U_2$
requires uniformity on the decay of the tail of the inducing times,
which is covered, in our case, by Lemma~\ref{lem:uniform decay}.
Condition $U_3$ demands that the constants involved on the estimates
of the induced map (such as bounded distortion, derivative growth,
backward contraction, etc.) can be chosen uniformly in a
neighbourhood of each map $f\in\mathcal U$. This also holds in the
present situation as it has been discussed during the proof of
Lemma~\ref{lem:uniform decay}.

Consequently, our inducing schemes and their properties put us
trivially in the setting of Alves and Viana, meaning that both
$\F_e(r, \ell, C,\alpha)$ and $\F_p(r, \ell, C,\beta)$ meet all the
requirements of the family $\mathcal U$ in
\cite[Therorem~A]{AV:2002}, from which we conclude that
\[
\F\ni f\mapsto \frac{d\mu_f}{dm}
\]
is continuous as in \eqref{eq:cty}, where $\F$ stands for either
$\F_e(r, \ell, C,\alpha)$ or $\F_p(r, \ell, C,\beta)$ and $m$ denotes
Lebesgue measure.
\end{proof}

\begin{remark}
Note that the theory presented here extends to Manneville-Pomeau maps
$f:x\mapsto x+x^{1+\alpha}\ ({\rm mod }\ 1)$ for $\alpha\in (0,1)$.
Given such a map, and a potential $\phi_t:=-t\log|Df|$, it is
straightforward to prove an equivalent of \cite[Theorem 1]{BTeqnat},
yielding an equilibrium state $\mu_t$ for $t\in [\delta,1]$ for some
$\delta<0$.  One main difference in proving statistical stability for
these measures is that in the proofs of Proposition~\ref{prop:Fn
fringe} and Lemma~\ref{lem:meas-prox-cylin} for example, to estimate
the measure of sets $\c_{0,k}^i\triangle \c_{n,k}^i$ we can no longer
assume that no two cylinders for the inducing schemes are adjacent.
Above, this property enabled us to estimate $\c_{0,k}^i\triangle
\c_{n,k}^i$ using the measure of 1-cylinders.  However, when, as in
the Manneville-Pomeau case, we do not have this property, we can use
the measure of $k$-cylinders to give us the required estimates
instead.

In broad terms, the theory presented here will also go through in more general families of maps, for example to simple generalisations of Manneville-Pomeau maps.  As we have seen, the important ingredients are that the inducing schemes for the families can be chosen so that for nearby maps, the inducing schemes are `close';  that the thermodynamic formalism in Section~\ref{sec:press} goes through for the inducing schemes; and that there are uniform bounds for the decay of the tail sets (as in Lemma~\ref{lem:uniform decay}) for all the inducing schemes.
\label{rmk:M-P}
\end{remark}

\medskip
\noindent Centro de Matem\'atica da Universidade do Porto\\ Rua do
Campo Alegre 687\\ 4169-007 Porto\\ Portugal\\
\texttt{jmfreita@fc.up.pt, http://www.fc.up.pt/pessoas/jmfreita/}\\
\texttt{mtodd@fc.up.pt, http://www.fc.up.pt/pessoas/mtodd/}

\end{document}